\documentclass[12pt,reqno]{amsart}

\usepackage[english]{babel}
\usepackage[utf8]{inputenc}
\usepackage{amsmath, amsthm, amssymb}
\usepackage{amsmath}
\usepackage{amsfonts}
\usepackage{amsthm}
\usepackage{mathrsfs}
\usepackage{amsrefs}
\usepackage{hyperref}
\usepackage{graphicx,epstopdf,color}
\usepackage{appendix}

\usepackage{xcolor}

\renewcommand{\leq}{\leqslant}
\renewcommand{\le}{\leqslant}
\renewcommand{\geq}{\geqslant}

\theoremstyle{plain}
\newtheorem{theorem}{Theorem}[section]
\newtheorem{corollary}[theorem]{Corollary}
\newtheorem{lemma}[theorem]{Lemma}
\newtheorem{proposition}[theorem]{Proposition}

\theoremstyle{definition}

\numberwithin{equation}{section}

\newcommand{\R}{\mathbb{R}}
\newcommand{\e}{\varepsilon}
\newcommand{\abs}[1]{\left\lvert#1\right\rvert}
\newcommand{\norm}[1]{\|{#1}\|}
\newcommand{\seminorm}[1]{\left[#1\right]}

\begin{document}
	
\title[Energy  asymptotics  of  a  Dirichlet  to  Neumann  problem]{Energy  asymptotics  of  a  Dirichlet  to  Neumann  problem  related to water waves}

\author[Pietro Miraglio]{Pietro Miraglio}

\address{P.M., Dipartimento di Matematica, Universit\`a di Milano, Via Cesare Saldini 50,
	20133 Milan, Italy, Departament de Matem\`{a}tica Aplicada I,
	Universitat Polit\`{e}cnica de Catalunya, Diagonal 647, 08028 Barcelona, Spain}
\email{pietro.miraglio@unimi.it}

\author[Enrico Valdinoci]{Enrico Valdinoci}
\address{E.V., Department of Mathematics and Statistics,
	University of Western Australia,
	35 Stirling Highway, WA6009 Crawley, Australia}
\email{enrico.valdinoci@uwa.edu.au}

\thanks{P.M. is supported by the MINECO grant MTM2017-84214-C2-1-P and is part of the Catalan research group 2017SGR1392. 
E.V. is supported by the Australian Research Council Discovery
Project DP170104880 NEW ``Nonlocal Equations at Work''.
The authors are members of INdAM-GNAMPA.
Part of this work was carried out on the occasion of a very pleasant visit of the first author to the University of Western Australia, which we thank for the warm hospitality.
}

\begin{abstract}
	We consider a Dirichlet to Neumann operator $\mathcal{L}_a$ arising in a model for water waves, with a nonlocal parameter $a\in(-1,1)$. We deduce the expression of the operator in terms of the Fourier transform, highlighting a local behavior for small frequencies and a nonlocal behavior for large frequencies. 
	
	We further investigate the $ \Gamma $-convergence of the energy associated to the equation $ \mathcal{L}_a(u)=W'(u) $, where $W$ is a double-well potential. When $a\in(-1,0]$ the energy $\Gamma$-converges to the classical perimeter, while for $a\in(0,1)$ the $\Gamma$-limit is a new nonlocal operator, that in dimension $n=1$ interpolates the classical and the nonlocal perimeter.
\end{abstract}

\maketitle

\section{Introduction}
In this article, we consider a possibly singular or degenerate elliptic problem with weights, which is
set on the infinite domain~$\R^n\times(0,1)$, endowed with mixed boundary conditions. When~$n=2$, such a problem is related to the formation of water waves from a steady ocean, the case of homogeneous density of the fluid corresponding to a Laplace equation in~$\R^2\times(0,1)$ with mixed boundary conditions, and the weighted equation arising from power-like fluid densities.

We provide here two types of results. The first set of results 
focuses on the operator acting on~$\R^n\times\{0\}$
produced by the associated Dirichlet to Neumann problem.
That is, we consider the weighted Neumann derivative
of the solution along the portion of the boundary that is endowed with a Dirichlet datum, which corresponds, in the homogeneous fluid case, to the determination of
the vertical velocity field on the surface of the ocean. In this setting, we
provide an explicit expression of this Dirichlet to Neumann operator in terms of the Fourier representation, and we
describe the asymptotics of the corresponding Fourier symbols. 

The second set of results deals with the energy functional
associated to the Dirichlet to Neumann operator. Namely,
we consider an energy built by the combination
of a suitably weighted interaction functional of Dirichlet to Neumann type in the Fourier space with a double-well potential.
In this setting, choosing the parameters in order to produce significant asymptotic
structures, we describe the $\Gamma$-limit configuration.
\medskip

The results obtained are new even in the case~$n=2$
and even for the Laplace equation. Interestingly, however,
the fluid density plays a decisive role as a bifurcation parameter, and the case of uniform density is exactly the threshold separating two structurally different behaviors. Therefore, understanding
the ``more general'' case of variable densities also provides structural
information on the homogeneous setting. Specifically, we prove convergence
of the energy functional to a $\Gamma$-limit corresponding to a mere interaction energy when~$a\in(0,1)$ and to the classical perimeter when~$a\in(-1,0]$. In terms of the corresponding fractional parameter~$s=\frac{1-a}2$, this dichotomy reflects a purely nonlocal behavior when~$s\in(0,1/2)$
and a purely classical asymptotics when~$s\in[1/2,1)$. Interestingly,
the threshold~$s=1/2$ corresponds here to the homogeneous density case,
the strongly nonlocal regime corresponds to degenerate densities~$y^a$ with~$a>0$,
and the weakly nonlocal regime to singular densities~$y^a$ with~$a<0$.\medskip

We also point out that the threshold~$s=1/2$
that we obtain here, as well as the limit behavior for
the regime~$s\in[1/2,1)$, is common to other
nonlocal problems,
such as the ones in~\cite{SV, SVdue, SOUG}. On the other hand, the limit functional
that we obtain in the strongly nonlocal regime~$s\in(0,1/2)$ appears to be new in the literature, and structurally different from other energy functionals of
nonlocal type that have been widely investigated.
\medskip

The precise mathematical formulation
of the problem under consideration is the following.	We consider the slab $\R^n\times[0,1]$ with coordinates $x\in\R^n$ and $y\in[0,1]$, a smooth bounded function $u:\R^n\to\R$, and its bounded
	extension $v$ in the slab~$\R^n\times [0,1]$, which is the bounded function satisfying the mixed boundary value
	problem
	\begin{align}
	\label{lvsis}
	\begin{cases}
	\mathrm{div}(y^a\nabla v)=0 \qquad &\text{in}\,\,\R^n\times (0,1)\\
	v_y(x,1)=0 \qquad &\text{on}\,\,\R^n\times \{y=1\}\\
	v(x,0)=u(x) \qquad &\text{on}\,\,\R^n\times \{y=0\},
	\end{cases}
	\end{align}
	where $a\in (-1,1)$.
	Problem~\eqref{lvsis}
	naturally leads to the study of the 
	Dirichlet to Neumann operator $\mathcal{L}_a$ defined as
	\begin{equation}
	\label{operator}
	\mathcal{L}_a u(x)=-\displaystyle\lim_{y\rightarrow 0}y^av_y(x,y).
	\end{equation}
	
	The operator $ \mathcal{L}_a $, which is the main object of the present work, arises in the study of a water wave model. With respect
	to the physical motivation, one can consider $ \R^n\times(0,1) $ as ``the sea'', where $\{y=0\}$ corresponds the surface of the sea (assumed to be at rest)
	and $\{y=1\}$ is its bottom (assumed to be made of concrete and impenetrable material). More specifically, the first equation in~\eqref{lvsis} models the mass conservation and the irrotationality of the fluid, and the second one is a consequence of the impenetrability of the matter. 
The scalar function~$v$ plays the role of a velocity potential, that is the gradient of~$v$
corresponds to the velocity of the fluid particles.
Given the datum of the velocity potential~$v$ on the surface --- i.e. the Dirichlet condition on $\{y=0\}$ in~\eqref{lvsis} --- we are interested in studying the weighted vertical velocity on the surface, which is responsible for the	formation of a wave emanating from the rest position of a ``flat sea". The operator $\mathcal{L}_a$ defined in~\eqref{operator} models indeed this vertical velocity.
	We refer to~\cite{DMV} for a complete description of this model
	and for detailed  physical motivations.\medskip
	 
We observe that the energy functional associated to~\eqref{lvsis} can be written as
	\[
	\mathcal{E}_K(v):=\frac12\int_{\R^n\times(0,1)}y^a\abs{\nabla v}^2\,dx\,dy.
	\]
	In what follows, we will consider the energy minimization
	in the class of functions 
	\begin{equation}\label{class}
	\begin{split}
	\mathcal{H}_u:=\{ w\in H^1_\text{loc}(\R^n\times(0,1),y^a)\,\,\text{s.t.}\,\,w(x,0)=u(x)
	\,\,
	\text{for a.e.}\,\,x\in\R^n \}.
	\end{split}
	\end{equation}
	Such a minimizer exists and it is unique --- see Lemma \ref{lemma_extension} below for a detailed proof --- and we can define the interaction energy associated to $u$ as the interaction energy of its minimal extension~$v$. Namely, with a slight abuse of notation, we write
	\[
	\mathcal{E}_K(u):=\inf_{v\in\mathcal{H}_u}\mathcal{E}_K(v).
	\]
	Notice that the minimizer $v\in\mathcal{H}_u$ of the energy $\mathcal{E}_K$ solves the mixed boundary problem \eqref{lvsis} in the weak sense, i.e.
	\begin{equation}\label{eq_weak}
	\int_{\R^n\times(0,1)}y^a\nabla v\cdot\nabla \varphi =0 \qquad 
	\end{equation}
 	for every $\varphi\in C^\infty(\R^n\times[0,1])$ with compact support contained in $\R^n\times(0,1]$.
 	
We observe that, thanks to the existence of a  unique minimizer of the energy $\mathcal{E}_K$ in the class $\mathcal{H}_u$, the operator $\mathcal{L}_a$ is actually
well-defined. Indeed, among all the (possibly many) solutions to \eqref{lvsis}, we can uniquely choose the one which minimizes $\mathcal{E}_K$ in $\mathcal{H}_u$, and define $\mathcal{L}_au$ as its weighted vertical derivative evaluated at $y=0$, according to~\eqref{operator}.\medskip
	
	In the case $a=0$, which corresponds to
	$v$ being the harmonic extension of $u$ in $\R^n \times (0,1)$, the operator $\mathcal{L}_a$ defined in~\eqref{operator} was considered by de la Llave and the second author in~\cite{DllV}.	In particular, they studied the equation
	\begin{equation}\label{a_0_equation}
		\mathcal{L}_0(u)=f(u)\qquad\text{in}\,\,\R^n,
	\end{equation}
	where $f\in C^{1,\beta}(\R)$, and $\mathcal{L}_0$ is the operator defined in~\eqref{operator} with $a=0$. 	
	The main result in~\cite{DllV} is a Liouville theorem for monotone solutions to \eqref{a_0_equation}, which leads in dimension $n=2$ to the one-dimensional symmetry of monotone solutions. 
	
	Some years later, this Liouville theorem has been generalized by Cinti and the authors of this paper~\cite{CMV}
	to stable\footnote{We say that a solution $u$ to \eqref{a_equation} is stable if the second variation of the associated energy is nonnegative definite at $u$. We also remind that, for this kind of problems, monotone solutions are stable --- see \cite{CMV}. Clearly, minimizing solutions to \eqref{a_equation} are also stable.} solutions to
	\begin{equation}\label{a_equation}
	\mathcal{L}_a(u)=f(u)\qquad\text{in}\,\,\R^n,
	\end{equation}
	where $f\in C^{1,\beta}(\R)$ and $a\in(-1,1)$.
	More precisely, in~\cite{CMV} the rigidity of monotone and minimizing solutions to \eqref{a_equation} is obtained in the case $n=3$ for every $a\in(-1,1)$. This is done by
	combining the Liouville theorem for stable solutions with some new energy estimates for monotone and minimizing solutions to \eqref{a_equation}.
	
	The problem of proving one-dimensional symmetry of some special classes of solutions to \eqref{a_equation} is strictly related to a conjecture of De Giorgi for the classical Allen-Cahn equation, and also to an analogue conjecture for the fractional Laplacian. 
	These conjectures are also related to a classical question posed
	by Gary W. Gibbons which originated from cosmological
	problems.
	We refer to the recent survey~\cite{DMV} for more details about these connections and for an outline of the most important recent results in these fields.
	
	In \cite{DllV} the operator $\mathcal L_0$ is written via Fourier transform as
	\begin{equation}\label{fourier_a0}
	\mathcal{L}_0u=\mathcal{F}^{-1}\left(\frac{e^{\abs{\xi}}-e^{-\abs{\xi}}}{e^{\abs{\xi}}+e^{-\abs{\xi}}}\abs{\xi}\widehat{u}(\xi)\right),
	\end{equation}
	where $\widehat{u}$ denotes the Fourier transform of $u$ and $\mathcal{F}^{-1}$ the inverse Fourier transform.
	
	{F}rom expression \eqref{fourier_a0}, one can easily observe that for large frequencies the Fourier symbol of $\mathcal L_0$ is asymptotic to $\abs{\xi}$, which is the Fourier symbol of the half-Laplacian
	(hence, the high-frequency wave formation is related, at least
	asymptotically, to the operator~$\sqrt{-\Delta}$).\medskip 
		
	The first main result of the present paper
	extends \eqref{fourier_a0} to every $ a\in(-1,1) $, providing the Fourier representation of the operator $\mathcal{L}_a$ for every value of the parameter~$a$
	in terms of special functions of Bessel type.

	\begin{theorem}
	\label{thm_fourier}
	For every smooth bounded function $u$ defined on $\R^n$ which is integrable, we can write the operator~$ \mathcal{L}_a $ defined in~\eqref{operator} via Fourier transform, as
	\begin{equation}
	\label{anon0}
	\widehat{\mathcal{L}_au}(\xi)=c_1(s)\frac{J_{1-s}(-i\abs{\xi})}{J_{s-1}(-i\abs{\xi})}\abs{\xi}^{2s}\widehat{u}(\xi),
	\end{equation}
	where $ 1-a=2s $, $J_k$ is the Bessel function of the first kind of order $k$, and
	\begin{equation}\label{c1}
	c_1(s):=i\left(\frac{1-i}{2}\right)^{4s-2}\frac{\Gamma(1-s)}{\Gamma(s)}.
	\end{equation}
	Moreover, the symbol
	\begin{equation}\label{S_s}
	S_s(\xi):= c_1(s)\frac{J_{1-s}(-i\abs{\xi})}{J_{s-1}(-i\abs{\xi})}\abs{\xi}^{2s}
	\end{equation}
	is a positive and increasing function of $ \abs{\xi} $, and enjoys the following asymptotic properties. There exist two positive constants $C_1$ and $C_2$ depending only on $s$ such that
	\begin{equation}\label{asympt}
	\begin{aligned}
	\lim_{\abs{\xi}\to0}\frac{S_s(\xi)}{\abs{\xi}^2}=C_1;
	\\
	\lim_{\abs{\xi}\to+\infty}\frac{S_s(\xi)}{\abs{\xi}^{2s}}=C_2.
	\end{aligned}
	\end{equation}
\end{theorem}

	We remind that $\abs{\xi}^2$ is the Fourier symbol of the classical Laplacian and that the fractional Laplacian can be expressed for a smooth function~$u$ defined in~$\R^n$ as
	\[
	\left(-\Delta \right)^s u(x)= \mathcal{F}^{-1}\left(\abs{\xi}^{2s}\widehat{u}(\xi)	\right).
	\]
	As a consequence, from Theorem \ref{thm_fourier} we have that the operator $\mathcal{L}_a$ defined in~\eqref{operator} is somewhat asymptotically related to the fractional Laplacian, but it is not equal to any purely fractional operator.
In this spirit,	
the asymptotic behaviors in~\eqref{asympt}
reveal an important difference between the problem considered here
and several other fractional problems widely investigated in the literature. Namely, in light of~\eqref{asympt}, we have that for large frequencies
the Fourier symbol of the operator~$\mathcal{L}_a$ is asymptotic to the Fourier symbol of the fractional
Laplacian~$(-\Delta)^s$ with~$s=\frac{1-a}2$, but for small frequencies
it is always asymptotic to the Fourier symbol of the classical Laplacian, and this lack of homogeneity, combined with a significant
structural difference ``between zero and infinity'',
suggests a new and interesting interplay between local and nonlocal phenomena
at different scales.\medskip

{F}rom \eqref{anon0} we also deduce an alternative formulation of the Dirichlet energy~$\mathcal{E}_K$, that we state in the following result.

\begin{corollary}\label{cor_equivalence}
	Let $u$ be a smooth bounded function defined on $\R^n$ which is integrable, and $v$ the solution of~\eqref{lvsis} obtained as the unique minimizer of~$\mathcal{E}_K$ in the class~$\mathcal{H}_u$. 
	Then,
\begin{equation}\label{fourier_energy}
	\mathcal{E}_K(v)=\frac12 \int_{\R^n\times(0,1)}y^a\abs{\nabla v}^2\,dx\,dy = \frac1{2(2\pi)^n}\int_{\R^n} S_s(\xi)\abs{\widehat{u}(\xi)}^2\,d\xi,
\end{equation}
	where $ a=1-2s$ and $S_s(\xi)$ is defined in~\eqref{S_s}.
\end{corollary}

	For later convenience, we introduce the notation 
	\begin{equation}\label{S_s_tilde}\begin{split}
&S_s(\xi)=\abs{\xi}^{2s}\widetilde{S}_s(\xi),\\{\mbox{where}}\qquad&
	\widetilde{S}_s(\xi):=c_1(s)\frac{J_{1-s}(-i\abs{\xi})}{J_{s-1}(-i\abs{\xi})},
	\end{split}\end{equation}
	and $c_1(s)$ is defined in~\eqref{c1}. When $s=1/2$, from \eqref{fourier_a0} we know that $\widetilde{S}_{1/2}$ is the hyperbolic tangent of $\abs{\xi}$. In general, $\widetilde{S}_s$ is expressed in terms of Bessel functions of the first kind, and its behavior at zero and at infinity can be easily deduced by~\eqref{asympt}. Indeed, $\widetilde{S}_s$ converges to a finite constant at infinity, while it behaves like $\abs{\xi}^{2-2s}$ near zero. This can be seen also in Figure \ref{figure_S_tilde}, where the plots of $\widetilde{S}_s$ are displayed for some values of $s\in(0,1)$. 

\begin{figure}[hbt] 
	\centering
	\includegraphics[width=12cm]{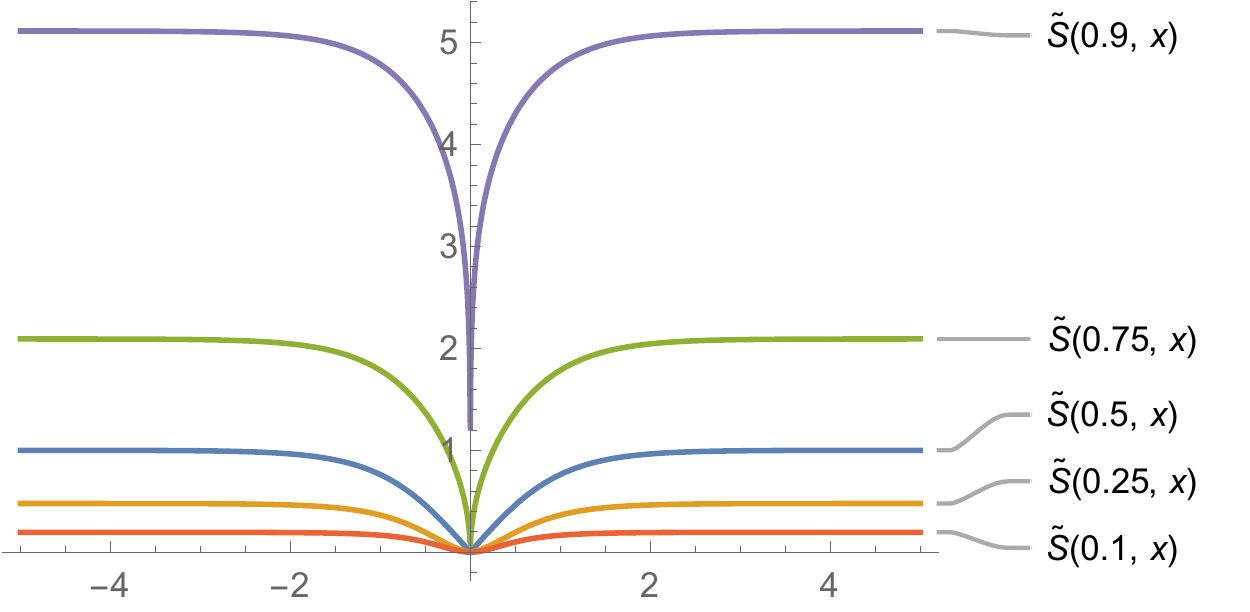}
	\caption{\em {The symbols $\widetilde{S}_s$ for different values of $s\in(0,1)$.}}
	\label{figure_S_tilde}
\end{figure}
	
	Heuristically, on the one hand,
	the connection of $\mathcal{L}_a$ with the fractional Laplacian was already evident from the formulation \eqref{lvsis}-\eqref{operator} of the operator, using the extension problem. Indeed, if we consider a solution $v$ of \eqref{lvsis} in the whole half-space and not only in a strip of fixed height, then the associated Dirichlet to Neumann operator is the fractional Laplacian $(-\Delta)^s$ with $s=(1-a)/2$ --- see \cite{CS}.

On the other hand,
the asymptotic properties outlined in~\eqref{asympt} make more clear the different nature of~$ \mathcal{L}_a $
in dependence of the parameter~$a$, which is a very specific feature of this operator. In order to further investigate this twofold behavior, we study the $\Gamma$-convergence of the energy associated to the equation $\mathcal{L}_au=W'(u)$, where $W$ is a double-well potential.

As well-known,
the $\Gamma$-convergence is a variational notion of convergence for functionals, which was introduced in~\cite{DG,DGF} and that captures the minimizing features of the energy --- see also \cite{M} for a classical example of $\Gamma$-convergence in the context of phase transitions.
In the recent years, there have been an increasing interest towards $\Gamma$-convergence results for nonlocal functionals, and some important results in this topic have been obtained, see for instance \cite{ABS,ABS1,Gon,ADM,SV,SVdue,SOUG}. For a complete introduction to topic of $\Gamma$-convergence, we refer the reader to \cite{DM,B}.

Since the operator $\mathcal{L}_a$ is strictly related to the fractional Laplacian, we are particularly interested in the paper \cite{SV} by Savin and the second author, in which they consider a proper rescaling of the energy
\begin{equation}\label{SV_functional}
\mathcal{I}_\e(u,\Omega):=\e^{2s}\mathcal{K}(u,\Omega)+\int_\Omega V(u)\,dx,
\end{equation}
where $V$ is a double-well potential, $\Omega$ a bounded set, and $\mathcal{K}(u,\Omega)$ is defined as
\begin{equation*}
\mathcal{K}(u,\Omega):=\iint_{\Omega\times\Omega}\frac{\abs{u(x)-u(y)}^2}{\abs{x-y}^{n+2s}}\,dx\,dy+2\iint_{\Omega\times\mathscr{C}\Omega}\frac{\abs{u(x)-u(y)}^2}{\abs{x-y}^{n+2s}}\,dx\,dy.
\end{equation*}
Observe that $\mathcal{K}(u,\Omega)$ is the ``$\Omega$-contribution''
of the $H^s$ seminorm of $u$, where
\begin{equation*}
[u]^2_{H^s(\R^n)}:=\iint_{\R^n\times\R^n}\frac{\abs{u(x)-u(y)}^2}{\abs{x-y}^{n+2s}}\,dx\,dy.
\end{equation*}

The main result in~\cite{SV} --- that we describe in more detail in Section \ref{sec_gamma_local} before the proof of Theorem \ref{thm_gammaconvergence} --- establishes that a proper rescaling of $\mathcal{I}_\e$ converges in the $\Gamma$-sense to the classical perimeter when $s\geq1/2$ and to the nonlocal area functional for~$s\in(0,1/2)$.

For some set $E\subset\R^n$, the nonlocal area functional of $\partial E$ in $\Omega$ is defined as $\mathcal{K}(u,\Omega)$ for $u=\chi_E-\chi_{\mathscr{C}E}$. 
This notion was introduced by Caffarelli, Roquejoffre and Savin in~\cite{CRS}, and takes into
account the interactions between points which lie in the set~$E$
and points which lie in its complement,
thus producing a functional which can be thought as a nonlocal version of the classical perimeter. For an introduction to this topic, we refer to \cite[Chapter 5]{BV}, \cite{CF}, \cite{DV}, and \cite{Luca}.

We also recall\footnote{See \cite[Proposition 3.4]{H} for the proof of \eqref{H_s_fourier}, and observe that $(2\pi)^{-n}$ is missing in the proof when they apply the Plancherel theorem.
} that the $H^s$ seminorm can be written via Fourier transform as
\begin{equation}\label{H_s_fourier}
[u]^2_{H^s(\R^n)}=\frac{2C(n,s)^{-1}}{(2\pi)^n}\int_{\R^n}\abs{\xi}^{2s}\abs{\widehat{u}(\xi)}^2\,d\xi,
\end{equation}
where 
\begin{equation}\label{C(n,s)}
C(n,s):=\left(\int_{\R^n}\frac{1-\cos(\zeta_1)}{\abs{\zeta}^{n+2s}}\,d\zeta\right)^{-1}.
\end{equation}

The alternative form \eqref{H_s_fourier} of the $H^s$ seminorm highlights the similarity between~$\mathcal{E}_K$ and the Dirichlet energy $\mathcal{K}(u,\Omega)$ in~\eqref{SV_functional}. This is evident after comparing \eqref{H_s_fourier} with expression \eqref{fourier_energy} for $\mathcal{E}_K$, taking also into account that the symbol $S_s(\xi)$ behaves like $\abs{\xi}^{2s}$ for high frequencies --- see \eqref{asympt}. 

This fact, together with the results in~\cite{SV}, leads to the natural question of studying the $\Gamma$-convergence of 
a proper rescaling of
\[
\mathcal{J}(u):=\int_{\R^n}S_s(\xi)\abs{\widehat{u}(\xi)}^2\,d\xi+\int_{\R^n}W(u)\,dx,
\]
where $W$ is a double-well potential $W(t)$. In particular, throughout the paper we assume that $W(t)$ satisfies
\begin{equation}\label{double_well_potential}
\begin{aligned}
W\in C^{2,\gamma}([0,1]), \qquad W(0)&=W(1)=0, \qquad W>0\,\,\text{in}\,\,(0,1), 
\\
W'(0)=W'(1)=0, \qquad&\text{and}\qquad W''(0)=W''(1)>0.
\end{aligned}
\end{equation}
Observe also that the fact of being a double-well potential is invariant under a multiplicative constant.

The energy functional $\mathcal{J}$ is similar to $\mathcal{I}_\e$ considered in~\cite{SV}, with the important structural
difference of replacing $\mathcal{K}(u,\Omega)$ with the Dirichlet energy associated to the operator $ \mathcal{L}_a $, expressed with the Fourier transform.

For every $ s\in(0,1) $, we consider the partial rescaling of $\mathcal{J}$ given by
\begin{equation}\label{Jfunctional}
\mathcal{J}_\e(u):=\e^{2s}\int_{\R^n}S_s(\xi)\abs{\widehat{u}(\xi)}^2\,d\xi+\int_{\R^n}W(u)\,dx.
\end{equation} 

We also define the function space in which we work as
\begin{equation}\label{1.19bis}
X:=\left\{	u\in L^\infty(\R^n) \,\,\text{s.t.}\,\,u\,\,\text{has compact support and}\,\,0\leq u\leq1	\right\},
\end{equation}
and we say that a sequence $u_j\in X$ converges to $u$ in $X$ if $u_j\to u$ in~$L^1(\R^n)$. Observe indeed that, according to the definition, $X\subset L^1(\R^n)$.

In order to obtain an interesting result in terms of $ \Gamma $-convergence, we take the rescaling of~\eqref{Jfunctional} given by $\mathcal{F}_\e:X\to\R\cup\{+\infty\}$, where
\begin{equation}\label{functional}
\mathcal{F}_\e(u):=\begin{cases}
\begin{aligned}
&\e^{-2s}\mathcal{J}_\e(u) \qquad &\text{if}\,\,s\in(0,1/2); \\
&\abs{\e\log\e}^{-1}\mathcal{J}_\e(u) \qquad &\text{if}\,\,s=1/2; \\
&\e^{-1}\mathcal{J}_\e(u) \qquad &\text{if}\,\,s\in(1/2,1).
\end{aligned}
\end{cases}
\end{equation}	

It is important to point out that the rescaling of $\mathcal{J}_\e$ that we consider here is the same as the one used for the functional $\mathcal{I}_\e$ in~\cite{SV}, and it is chosen to produce a significant $ \Gamma $-limit
from the interplay of interaction and potential energies.

When $ s\in(0,1/2) $, the limit functional $ \mathcal{F}:X\to\R\cup\{+\infty\} $ is defined as
\begin{equation}\label{nonlocal_limit}
\mathcal{F}(u):=\begin{cases}
\begin{aligned}
&\int_{\R^n} S_s(\xi)\abs{\widehat{u}(\xi)}^2\,d\xi \qquad &\text{if}\,\,u=\chi_E,\,\,\text{for some set}\,\,E\subset\R^n; \\
&+\infty \qquad &\text{otherwise}.
\end{aligned}
\end{cases}
\end{equation}
We point out that the limit functional $\mathcal{F}$ for $ s\in(0,1/2) $ is well-defined when $u=\chi_E$. This is a consequence of the fact that its difference with the $H^s$ seminorm of~$u=\chi_E$ is finite --- see the forthcoming Lemma \ref{lemma_difference} --- and that the nonlocal area functional of a bounded set is always well-defined for $s\in(0,1/2)$. Moreover, as stated explicitly in
Lemma \ref{lemma_difference} means that we can see $\mathcal{F}$ as a perturbation of the nonlocal area functional.
We will further comment on the functional $\mathcal{F}$ for $ s\in(0,1/2) $ in Proposition~\ref{prop_asympt} below.

In the case $ s\in[1/2,1) $, we define $ \mathcal{F}:X\to\R\cup\{+\infty\} $ as
\begin{equation}\label{local_limit}
\mathcal{F}(u):=\begin{cases}
\begin{aligned}
&c_\#\textnormal{Per}(E) \qquad &\text{if}\,\,u=\chi_E\,\,\,\text{for some set}\,\,\,E\subset\R^n; \\
&+\infty \qquad &\text{otherwise},
\end{aligned}
\end{cases}
\end{equation}
where $c_\#$ is a positive constant depending only on $n$ and $s$, and $\textnormal{Per}(E)$ denotes the classical
perimeter of the set $E$, in the sense described e.g. in~\cite{G}.\medskip

The following is the second main result of the present paper. It establishes the $\Gamma$-convergence of the rescaled functional~\eqref{functional} to $ \mathcal{F} $ defined in~\eqref{nonlocal_limit}-\eqref{local_limit}.

\begin{theorem}\label{thm_gammaconvergence}
	Let $ s\in(0,1) $. Then the functional $ \mathcal{F}_\e $ defined in~\eqref{functional} $ \Gamma $-converges to the functional~$ \mathcal{F} $ defined in~\eqref{nonlocal_limit}-\eqref{local_limit}, i.e. for any $u$ in $X$
	\begin{itemize}
		\item[(i)] for any $ u_\e $ converging to $u$ in $X$
		\begin{equation}\label{liminf}
		\liminf_{\e\to0^+}\mathcal{F}_\e(u_\e)\geq\mathcal{F}(u);
		\end{equation}
		\item[(ii)] there exists a sequence $ (u_\e)_\e $ converging to $ u $ in $ X $ such that
		\begin{equation}\label{limsup}
		\limsup_{\e\to0^+}\mathcal{F}_\e(u_\e)\leq\mathcal{F}(u).
		\end{equation}
	\end{itemize}
\end{theorem}
We stress that the $\Gamma$-limit functional $\mathcal{F}$ is defined in two different ways depending on whether $s$ is above or below $1/2$, showing a purely local
behavior when~$s\in[1/2,1)$ and a purely nonlocal behavior when~$s\in(0,1/2)$. In view
of the different structure of the problem
in terms of the nonlocal parameter~$s$, we prove Theorem~\ref{thm_gammaconvergence}
in two different ways depending on the parameter range. For $s\in[1/2,1)$ the proof is presented in Section~\ref{sec_gamma_local}, while for $s\in(0,1/2)$ we include it in Section~\ref{sec_nonlocal}.
 
When $s\in[1/2,1)$, we recover the classical perimeter in the $\Gamma$-limit, as in the case of the energy associated to the fractional Laplacian treated in~\cite{SV}. 
Moreover, the result in~\cite{SV} plays a key role in our proof of Theorem \ref{thm_gammaconvergence} for $s\geq1/2$. Indeed, in this case we ``add and subtract'' the square of the $H^s$-seminorm --- properly rescaled --- to the functional $\mathcal{F}_\e$. In this way, we write $\mathcal{F}_\e$ as the nonlocal area functional plus a remainder term. We then show that the remainder term goes to zero in the limit, and deduce the proof of Theorem \ref{thm_gammaconvergence} for $s\in[1/2,1)$ from a proper application of~\cite[Theorem 1.4]{SV}.

On the other hand, when $s\in(0,1/2)$, the $\Gamma$-limit is the functional $\mathcal{F}$ defined in~\eqref{nonlocal_limit}, that has a nonlocal feature. 
As a technical remark, we also point out that, in our framework, the case~$s\in[1/2,1)$ is conceptually harder to address than the case~$s\in(0,1/2)$,
and the computational complications arising when~$s\in[1/2,1)$ are often motivated by the fact that one has to relate a nonlocal behavior at a given configuration with a local asymptotic pattern.\medskip

When $n=1$, we are able to make explicit computations with the Fourier transform, and obtain additional information on the $\Gamma$-limit functional $\mathcal{F}$ defined in~\eqref{nonlocal_limit}.

To this end,
since the limit functional $ \mathcal{F} $ is (possibly) finite only when $ u=\chi_E $ for some set $ E\subset\R $, we consider a connected interval $I_r\subset\R$ of length $r$ and the characteristic function $\chi_{I_r}$. Then, the squared modulus of the Fourier transform of~$\chi_{I_r}$~is
\[
\abs{\widehat{\chi_{I_r}}(\xi)}^2=\frac{4\sin^2(r\xi)}{\xi^2}.
\]
For the sake of completeness
we included this computation in the appendix --- see Lemma~\ref{lemma_onebump}. We also
remark that the squared modulus of the Fourier transform of $\chi_{I_r}$ depends only on the length of the interval, thus $ \mathcal{F}(\chi_{I_r}) $ only depends on $r$. 

Therefore, we can define a function $ \mathcal{T}_s(r):[0,+\infty)\longrightarrow[0,+\infty) $ as
\begin{equation}\label{intro_T_s}
\mathcal{T}_s(r):=\mathcal{F}(\chi_{I_r})=\int_\R S_s(\xi)\abs{\widehat{\chi_{I_r}}(\xi)}^2\,d\xi,
\end{equation}
	where $ I_r\subset\R$ is a connected interval of length $r$. Observe that $ \mathcal{T}_s $ depends on $s\in(0,1/2)$, as the symbol $ S_s(\xi) $ defined in~\eqref{S_s} depends on $s$. The following result contains some properties of the function $\mathcal{T}_s$ that allow us to relate it to the common notions of classical and fractional perimeter in one dimension.
	
\begin{proposition}\label{prop_asympt}
	Let $ s\in(0,1/2)$ and $n=1$. The function $ \mathcal{T}_s(r) $ defined in~\eqref{intro_T_s} is positive and enjoys the following asymptotic properties. There exist two positive constants $C_1$ and $C_2$ depending only on $s$ such that
	\begin{align}
	\label{asynt_0}
	&\lim_{r\to0^+}\frac{\mathcal{T}_s(r)}{r^{1-2s}}=C_1; 
	\\
	&\label{asynt_infty}
	\lim_{r\to+\infty}\mathcal{T}_s(r)=C_2.
	\end{align}
\end{proposition}

We recall that from the definition of nonlocal perimeter it follows that an interval of length $ r $ has fractional perimeter of order $ r^{1-2s} $. In this sense, Proposition~\ref{prop_asympt} tells us that the limit functional defined in~\eqref{nonlocal_limit} interpolates the classical and the fractional perimeter, at least in dimension one. Indeed, for intervals of small length ${\mathcal{T}_s(r)}$ behaves like the fractional perimeter, while for large values of $r$ it converges to a constant, counting the finite number of discontinuities of $\chi_{I_r}$.\medskip

We remark that the restriction $n=1$ in Proposition~\ref{prop_asympt}
is only due to the possibility of making explicit calculations with the Fourier transform. For this reason, we think that it is an interesting question
to understand how the functional $\mathcal{F}$ defined in~\eqref{nonlocal_limit} for $s\in(0,1/2)$ interpolates classical and nonlocal objects in any dimension.
\vspace{3mm}

\subsection*{Structure of the paper} In Section~\ref{sec_extension} we prove that there exists a unique minimizer of the energy $ \mathcal{E}_K $ in the class $ \mathcal{H}_u $. In Section~\ref{sec_ft} we prove Theorem~\ref{thm_fourier} about the Fourier representation of the operator $\mathcal{L}_a$. In Section~\ref{sec_gamma_local} we prove the $\Gamma$-convergence result of Theorem~\ref{thm_gammaconvergence} when $s\geq1/2$. In Section~\ref{sec_nonlocal} we assume $s\in(0,1/2)$ and we prove both Theorem \ref{thm_gammaconvergence} and  Proposition \ref{prop_asympt} about the limit functional. 

Finally, we collect in the appendix some ancillary
computations and technical results.

\section{Existence and uniqueness of the minimizer for the Dirichlet energy}\label{sec_extension}

This section concerns the existence and the uniqueness of the minimizer of the energy $ \mathcal{E}_K $ in the class of functions $ \mathcal{H}_u $ defined in~\eqref{class} for a given a smooth function~$u$. We state the existence and uniqueness result as follows.

\begin{lemma}\label{lemma_extension}
If $u$ is a bounded smooth function defined in $\R^n$, then there exists a unique minimizer of the functional $ \mathcal{E}_K $ in the class $ \mathcal{H}_u $.	
\end{lemma}
\begin{proof}
	\textbf{Step 1.} First, using a classical convexity argument,
	we prove that if such a minimizer exists, then it is \textit{unique}. If we assume that $v$ and $w$ are two minimizers of $ \mathcal{E}_K $ in $ \mathcal{H}_u $, then considering the energy of their arithmetic mean we find that
	\begin{equation}\label{76}
	\begin{split}
	\mathcal{E}_K\left(\frac{v+w}{2}\right)&=\frac12\int_{\R^n\times(0,1)}y^a\frac{\abs{\nabla v}^2+\abs{\nabla w}^2+2\nabla v\cdot\nabla w}{4}\,dx\,dy 
	\\
	&\leq\frac12\mathcal{E}_K(v)+\frac12\mathcal{E}_K(w) = \mathcal{E}_K(v).
	\end{split}
	\end{equation}
	Since $v$ and $w$ are minimizers for $ \mathcal{E}_k $, the Cauchy-Schwarz inequality in~\eqref{76} is an equality, hence
	\[
	\nabla v = \lambda \nabla w.
	\]
	Now, since $ \mathcal{E}_K(v)=\mathcal{E}_K(w) $, then $ \lambda=\pm1 $. If $ \lambda=+1 $, then $ v $ and $ w $ are equal up to an additive constant, but this constant must be zero since both functions are equal to $u(x)$ when~$ y=0 $. If instead $ \lambda=-1 $, then from~\eqref{76} we deduce that $ \mathcal{E}_K(v)=\mathcal{E}_K(w)=0 $, therefore $v$ and $w$ are constant, and these constants must coincide since they agree when $ y=0. $
	
	\textbf{Step 2.} Let us now prove \textit{existence}. First, we observe that this is equivalent to proving that there exists a minimizer of the energy
	\[
	\mathcal{E}_{K,2}(v):=\frac12\int_{\R^n\times(0,2)}y^a\abs{\nabla v}^2\,dx\,dy,
	\]
	in the class of functions
	\begin{equation*}
	\mathcal{H}_{u,2}:=\{ w\in H^1_\text{loc}(\R^n\times(0,2),y^a)\,\,\text{s.t.}\,\,w(x,0)=w(x,2)=u(x)\,\,\text{for a.e.}\,\,x\in\R^n \}.
	\end{equation*}
	
	Indeed, let us suppose for the moment that such a minimizer exists and let us denote it with $\overline{v}$. Then, we can deduce that it is unique, using the same argument as in Step 1. 
	
	Furthermore, since $ \overline{v} $ is a minimizer, then it is symmetric with respect to $\{y=1\}$. To see this, let us consider the competitor
	\begin{equation*}
	\widetilde{v}(x,y):=\begin{cases}
	\begin{aligned}
	&\overline{v}(x,y) \qquad &\text{if}\,\,0<y<1 \\
	&\overline{v}(x,2-y) \qquad &\text{if}\,\,1<y<2,
	\end{aligned}
	\end{cases}
	\end{equation*}	
	for which we have $\mathcal{E}_{K,2}(\widetilde{v})=\mathcal{E}_{K,2}(\overline{v})$ and $ \widetilde{v}\in\mathcal{H}_{u,2}$. By the uniqueness of the minimizer of $ \mathcal{E}_{K,2} $ in $ \mathcal{H}_{u,2} $, we deduce that $ \widetilde{v}\equiv\overline{v} $, and therefore that $\overline{v}$ is symmetric with respect to~$\{y=1\}$.
	Now, if we consider the restriction $\overline{v}_{|\R^n\times(0,1)}$, then it belongs to $\mathcal{H}_{u}$. In addition, using the minimality and symmetry properties of $\overline{v}$, we deduce by a reflection argument that $\overline{v}_{|\R^n\times(0,1)}$ minimizes $ \mathcal{E}_K $ in $ \mathcal{H}_u $.
	
	Summarizing, to prove Lemma \ref{lemma_extension} we are reduced to show that
	\begin{equation}\label{minimizer_2}
		\text{there exists a minimizer of the energy} \,\,\,  \mathcal{E}_{K,2} \,\,\,\text{in the class}\,\,\, \mathcal{H}_{u,2}. 
	\end{equation}

	In order to prove \eqref{minimizer_2}, we minimize the localized functional $\mathcal{E}_{K,2}$ on $ B_R\times(0,2) $ and then take the limit as $R\to+\infty$. More precisely, we want to prove that there exists a minimizer of
	\[
	\mathcal{E}_{K,2}^R(v):=\frac12\int_{B_R\times(0,2)}y^a\abs{\nabla v}^2\,dx\,dy,
	\]
	in the space
	\begin{equation*}
		\mathcal{H}^R_{u,2}:=\{ w\in H^1(B_R\times(0,2),y^a)\,\,\text{s.t.}\,\,w(x,0)=w(x,2)=u(x)
		\,\,\text{for a.e.}\,\,x\in B_R	\},
	\end{equation*}
	and then take the limit as $R\to\infty$.
	
	The existence of local minimizers for this problem follows from classical tools in the calculus of variations. Indeed, the lower boundedness of $ \mathcal{E}^R_{K,2} $ and the convexity with respect to the gradient give the weak lower semi-continuity of the functional --- see \cite[Theorem 1, p.446]{E}. In addition, $ \mathcal{E}^R_{K,2} $ is coercive\footnote{As a technical observation,
	we point out that the coercivity in this setting
	follows from the Poincar\'e inequality with Muckenhoupt weights --- see \cite[Chapter 15]{HKM}. We also observe that, for this inequality to hold, it is enough to assume the Dirichlet datum on a portion of the boundary with nonnegative Hausdorff measure.} in the $ H^1(B_R\times(0,2),y^a) $-norm and this, together with weak lower semicontinuity, is enough to conclude the existence of a minimizer of~$\mathcal{E}_{K,2}^R$ in the class $ \mathcal{H}^R_{u,2} $.
	
	Furthermore, the local minimizer is unique for every $R>0$, again by the standard convexity argument of Step 1. Therefore, for every $R>0$ we know that there exists a unique minimizer $v_R$ of $ \mathcal{E}^R_{K,2} $ in  $ \mathcal{H}_{u,2} $ and we want to deduce~\eqref{minimizer_2}, passing to the limit as $R\to\infty$. 
	
	To this end, we first observe that $v_S$ solves $\textnormal{div}(y^a\nabla v_S)=0$ in the weak sense in~$C_R$, whenever~$S\geq R$. We choose $ \varphi=v_S\eta^2 $ in the weak formulation~\eqref{eq_weak} of the equation, where $ \eta\in C^\infty_c(C_R,[0,1]) $ and $\eta\equiv1$ in $C_{R/2}$. Using also a Cauchy-Schwarz inequality, we obtain the Caccioppoli bound
	\begin{equation}\label{caccioppoli}
	\int_{B_{R/2}\times(0,2)}y^a\abs{\nabla v_S}^2\leq C\int_{B_{R/2}\times(0,2)}y^a\abs{v_S}^2,
	\end{equation}
	for a constant $C$ depending only on $R$.
	
	We then observe that, thanks to the maximum principle, every minimizer $v_S$ of the energy functional attains its maximum at a boundary point. This maximum has to be less or equal than $\norm{u}_{L^\infty(\R^n)}$, where $u$ is the Dirichlet datum on the top and the bottom of the cylinder. Indeed, if this is not the case, then we can build a competitor with lower energy than $v_S$ by simply truncating $v_S$ when its absolute value exceeds $\norm{u}_{L^\infty(\R^n)}$.
	
	Therefore, we can bound the right-hand side of \eqref{caccioppoli} with a constant depending only on $n$, $R$ and $\norm{u}_{L^\infty(\R^n)}$. This gives a uniform bound on the $H^1(B_{R/2}\times(0,2),y^a)$-norm of $v_S$ for every $S>R$. Hence, we can find a subsequence of $(v_S)$ that converges locally to a function $\overline{v}\in\mathcal{H}_{u,2}$. Finally, $\overline{v}$ minimizes $\mathcal{E}_{K,2}$ in $\mathcal{H}_{u,2}$ since $v_S$ are local minimizers, and this proves~\eqref{minimizer_2}. This concludes the proof of Lemma~\ref{lemma_extension}.
\end{proof}

\section{The energy via Fourier transform}
\label{sec_ft}
	In this section we want to prove the representation via Fourier transform of the operator $\mathcal{L}_a$, outlined in Theorem~\ref{thm_fourier}. We start by considering the simplest case $a=0$. To this end,
we observe that problem~\eqref{lvsis} with $ a=0 $ reads
	\begin{align}
	\label{a0system}
	\begin{cases}
	\Delta v=0\quad&{\mbox{ in }}\R^n\times(0,1)\\
	\partial_y v=0\quad&{\mbox{ on }}\R^n\times\{y=1\} \\
	v(x,y)=u(x)\quad&{\mbox{ on }}\R^n\times\{y=0\},
	\end{cases}
	\end{align}
	and the Dirichlet to Neumann operator is
	\begin{equation}\label{a0_operator}
	\mathcal{L}_0 u=-\partial_yv(x,y)_{\lvert_{\{y=0\}}}.
	\end{equation}
	
	In this case, the representation via Fourier transform already appears in~\cite{DllV} by de la Llave and the second author. We state here explicity
	this result and give a simple proof of it. We will then use the same strategy, combined
	with a suitable special functions analysis, to prove Theorem~\ref{thm_fourier} in the general case $a\in(-1,1)$.
	\begin{proposition}[de la Llave, Valdinoci \cite{DllV}]\label{prop_a0}
	For every smooth bounded function $u$ defined on $\R^n$ which is integrable, we can write the operator $ \mathcal{L}_0 $ defined in~\eqref{a0_operator} via Fourier transform as 
	\begin{equation}
	\label{a0}
	\widehat{\mathcal{L}_0u}=S_{1/2}(\xi)\widehat{u}(\xi)=\frac{e^{\abs{\xi}}-e^{-\abs{\xi}}}{e^{\abs{\xi}}+e^{-\abs{\xi}}}\abs{\xi}\widehat{u}(\xi).
	\end{equation}
	\end{proposition}
	\begin{proof}
	Taking the Fourier transform of the first equation in~\eqref{a0system}, we find an ODE in the variable $y$, that is
	\[
	-\lvert\xi\rvert^2\widehat{v}+\widehat{v}_{yy}=0.
	\]
	This equation is solved by
	\[
	\widehat{v}(\xi,y)=\alpha(\xi)e^{\lvert\xi\rvert y}+\beta(\xi)e^{-\lvert\xi\rvert y},
	\]
	where $\alpha$ and $\beta$ are functions depending only on $\xi$. In order to determine~$\alpha$ and~$\beta$, we consider the Fourier transform of the second and third equations in~\eqref{a0system}. The Dirichlet condition on $\{y=0\}$ gives
	\[
	\alpha(\xi)+\beta(\xi)=\widehat{u}(\xi),
	\]
	while the Neumann condition on $\{y=1\}$ gives
	\[
	\alpha(\xi)\lvert\xi\rvert e^{\lvert\xi\rvert}-\beta(\xi)\lvert\xi\rvert e^{-\lvert\xi\rvert}=0.
	\]
	Therefore, we find
	\begin{equation*}
	\alpha(\xi)=\frac{e^{-2\lvert\xi\rvert}}{1+e^{-2\lvert\xi\rvert}}\widehat{u}(\xi)
	\qquad
	\text{and}
	\qquad
	\beta(\xi)=\frac{1}{1+e^{-2\lvert\xi\rvert}}\widehat{u}(\xi).
	\end{equation*}
	Finally, computing the Fourier transform of $\mathcal{L}_0 u$, we find
	\begin{equation*}
	\widehat{\mathcal{L}_0u}(\xi)=-\partial_y\widehat{v}(\xi,y)_{\lvert_{\{y=0\}}}=\left(\beta(\xi)- \alpha(\xi)\right)\abs{\xi}=\frac{e^{\lvert\xi\rvert}-e^{-\lvert\xi\rvert}}{e^{\lvert\xi\rvert}+e^{-\lvert\xi\rvert}}\abs{\xi}\widehat{u}(\xi),
	\end{equation*}
	and this proves~\eqref{a0}.
	\end{proof}

	Now, we consider problem~\eqref{lvsis} for a general parameter~$a\in(-1,1)$ and we complete the proof of Theorem~\ref{thm_fourier}. For this,
	we use the same strategy as in the proof of Proposition~\ref{prop_a0}, but extra computations are required, together with a set of useful identities involving special functions.
	
	\begin{proof}[Proof of Theorem~\ref{thm_fourier}]
	As we did in the case $a=0$, we start by considering the Fourier transform of the first equation in~\eqref{lvsis}, that is
	\[
	-\lvert\xi\rvert^2y^a\widehat{v}+ay^{a-1}\widehat{v}_y+y^a\widehat{v}_{yy}=0.
	\]
	This is an ODE with respect to the variable $y$ and it is solved by
	\[
	\widehat{v}(\xi,y)=\alpha(\xi)y^{\frac{1-a}{2}}J_{{\frac{a-1}{2}}}(-i\lvert\xi\rvert y)+\beta(\xi)y^{\frac{1-a}{2}}Y_{{\frac{a-1}{2}}}(-i\lvert\xi\rvert y),
	\]
	where $J_m$ and $Y_m$ are Bessel functions of order $m$ of the first and second kind respectively, while $\alpha$ and $\beta$ are functions depending only on $\xi$.
		
	In order to determine $\alpha(\xi)$ and $\beta(\xi)$, we consider the Fourier transform of the second and third equations in~\eqref{lvsis}. The equation on $\{y=0\}$ gives
	
	\begin{equation}\label{fourier_eq1}
	\widehat{u}(\xi)=\alpha(\xi)\lim_{y\to0}y^{\frac{1-a}{2}}J_{\frac{a-1}{2}}(-i\lvert\xi\rvert y)+\beta(\xi)\lim_{y\to0}y^{\frac{1-a}{2}}Y_{\frac{a-1}{2}}(-i\lvert\xi\rvert y).
	\end{equation}
	We recall the two following properties of Bessel functions
	\begin{equation}
	\label{propJ}
	\lim_{x\to0}\frac{J_m(-ix)}{x^m}
	=\frac{2^{-2m}(1-i)^{2m}}{\Gamma(m+1)};
	\end{equation}
	\begin{equation}\label{bessel_prop}
	\text{for non integer $m$,}\qquad Y_m(x)=\frac{J_m(x)\cos(m\pi)-J_{-m}(x)}{\sin(m\pi)}.
	\end{equation}
	Now, using \eqref{propJ} and \eqref{bessel_prop}, we can write \eqref{fourier_eq1} as
	\begin{equation*}
	\widehat{u}(\xi)=\alpha(\xi) (1-i)^{a-1} \frac{2^{1-a}}{\Gamma(\frac{a+1}{2})}\abs{\xi}^\frac{a-1}{2}+\beta(\xi)(1-i)^{a-1} \frac{2^{1-a}}{\Gamma(\frac{a+1}{2})}\frac{\cos\big(\frac{a-1}{2}\pi\big)}{\sin\big(\frac{a-1}{2}\pi\big)}\abs{\xi}^\frac{a-1}{2}.
	\end{equation*}
	Using the relation $1-a=2s$, the equation on $\{y=0\}$ can be finally written as
	\begin{equation}\label{eq1}
		\abs{\xi}^s\widehat{u}(\xi)=\frac{(1-i)^{-2s}2^{2s}}{\Gamma(1-s)}\left\{ \alpha(\xi)-\frac{\cos\left(s\pi\right)}{\sin\left(s\pi\right)}\beta(\xi) \right\}.
	\end{equation}
	
	Now, we want to use the equation on $\{y=1\}$. First, we compute the derivative of~$\widehat{v}(\xi,y)$ with respect to $y$
	\begin{equation*}
		\begin{split}
			&\hspace{-0.5cm}\partial_y\widehat{v}(\xi,y)=\alpha(\xi)\frac{1-a}{2}y^\frac{-1-a}{2}J_{\frac{a-1}{2}}(-i\lvert\xi\rvert y)-\alpha(\xi)y^{\frac{1-a}{2}}i\lvert\xi\rvert J_{\frac{a-1}{2}}'(-i\lvert\xi\rvert y) 
			\\
			&\hspace{2cm}+\beta(\xi)\frac{1-a}{2}y^\frac{-1-a}{2}Y_{\frac{a-1}{2}}(-i\lvert\xi\rvert y)-\beta(\xi)y^{\frac{1-a}{2}}i\lvert\xi\rvert Y_{\frac{a-1}{2}}'(-i\lvert\xi\rvert y).
		\end{split}
	\end{equation*}
	We can simplify this expression using the following formulas for the derivatives of Bessel functions
	\[
	J_{\frac{a-1}{2}}'(x)=\frac{a-1}{2x}J_{\frac{a-1}{2}}(x)-J_{\frac{a+1}{2}}(x),
	\]
	\begin{equation*}
	Y_{\frac{a-1}{2}}'(x)=\frac{a-1}{2x}Y_{\frac{a-1}{2}}(x)-Y_{\frac{a+1}{2}}(x).
	\end{equation*}
	This gives
	\begin{equation}
		\label{dy}
		\partial_y\widehat{v}(\xi,y)=\alpha(\xi)i\lvert\xi\rvert y^\frac{1-a}{2}J_{\frac{a+1}{2}}(-i\lvert\xi\rvert y)+\beta(\xi)i\lvert\xi\rvert y^\frac{1-a}{2}Y_{\frac{a+1}{2}}(-i\lvert\xi\rvert y).
	\end{equation}
	Using again the relation $ 1-a=2s $, we write the Neumann condition over $ \{y = 1\} $ as
	\begin{equation}\label{eq2}
		0=J_{1-s}(-i\lvert\xi\rvert) \alpha(\xi)+ Y_{1-s}(-i\lvert\xi\rvert) \beta(\xi).
	\end{equation}
	To determine $\alpha$ and $\beta$, we put together the information given by~\eqref{eq1} and~\eqref{eq2} --- which are deduced from the second and third equation in~\eqref{lvsis}. In this way, we obtain the system
	\begin{equation}
	\label{alphabeta}
		\begin{cases}
			J_{1-s}(-i\lvert\xi\rvert) \alpha(\xi)+ Y_{1-s}(-i\lvert\xi\rvert) \beta(\xi)=0 \\ 
			\alpha(\xi)-\frac{\cos(s\pi)}{\sin(s\pi)}\beta(\xi)=\left(\frac{1-i}{2}\right)^{2s}\Gamma(1-s)\abs{\xi}^s\,\widehat{u}(\xi).
		\end{cases}
	\end{equation}
	Solving~\eqref{alphabeta}, we find
	\begin{equation}\label{beta}
	\begin{split}
	\alpha(\xi)&=-\widetilde{c}(s)\frac{Y_{1-s}(-i\lvert\xi\rvert)}{\cos\left(s\pi\right) J_{1-s}(-i\lvert\xi\rvert)+\sin\left(s\pi\right) Y_{1-s}(-i\lvert\xi\rvert)}\abs{\xi}^{s}\widehat{u}(\xi),
	\\
	\beta(\xi)&=\widetilde{c}(s)\frac{J_{1-s}(-i\lvert\xi\rvert)}{\cos\left(s\pi\right) J_{1-s}(-i\lvert\xi\rvert)+\sin\left(s\pi\right) Y_{1-s}(-i\lvert\xi\rvert)}\abs{\xi}^{s}\widehat{u}(\xi)
	\end{split}
	\end{equation}
	where
	\begin{equation*}
	\widetilde{c}(s):=-\left(\frac{1-i}{2}\right)^{2s}\sin(s\pi)\Gamma(1-s).
	\end{equation*}
	Using formula~\eqref{dy} for the $y$-derivative of $\widehat{v}$, we can compute the Fourier transform of $\mathcal{L}_a u$ and find
	\begin{equation}\label{transf:1}
		\begin{split}
			\widehat{\mathcal{L}_au}(\xi)&=-y^a\partial_y\widehat{v}(\xi,y)_{\lvert_{\{y=0\}}} \\
			&=-i\lvert\xi\rvert\bigg[\alpha(\xi)\lim_{y\to0}  y^{1-s}J_{1-s}(-i\lvert\xi\rvert y)+\beta(\xi)\lim_{y\to0} y^{1-s}Y_{1-s}(-i\lvert\xi\rvert y)\bigg].
		\end{split}
	\end{equation}
	Using the properties in~\eqref{propJ}-\eqref{bessel_prop} of Bessel functions, we see that the first limit in~\eqref{transf:1} is zero, and the second one gives a nontrivial contribution. More specifically, we have that
	\begin{equation*} 
			\widehat{\mathcal{L}_au}(\xi)=\frac{i}{\sin(s\pi)\Gamma(s)} \left(\frac{1-i}{2}\right)^{2s-2}
			\abs{\xi}^s\beta(\xi).
	\end{equation*}
	We can simplify this expression, also using~\eqref{bessel_prop} in~\eqref{beta}, and write it as
	\begin{equation*}
		\widehat{\mathcal{L}_au}(\xi)=
		c_1(s)\frac{J_{1-s}(-i\abs{\xi})}{J_{s-1}(-i\abs{\xi})}\abs{\xi}^{2s}\widehat{u}(\xi),
	\end{equation*}
	where 
	\[
	c_1(s)=i\left(\frac{1-i}{2}\right)^{4s-2}\frac{\Gamma(1-s)}{\Gamma(s)}.
	\]
	This proves \eqref{anon0}, and we are left with showing the asymptotic properties~\eqref{asympt} of the symbol $S_s(\xi)$ defined in~\eqref{S_s}. 
	
	First, we recall the notation in~\eqref{S_s_tilde}. 
	{F}rom the Taylor expansion near 0 of the
	Bessel functions of the first kind expressed in~\eqref{propJ}, we easily deduce that $\widetilde{S}_s(0)=0$ and
	\[
	\lim_{\abs{\xi}\to0}\frac{S_s(\xi)}{\abs{\xi}^2}=C_1,
	\] 
	where $C_1$ is a positive constant depending only on $s$. Moreover, $\widetilde{S}_s(\xi)$ is radially monotone increasing, since
\begin{equation}\label{der_Stilde}
	\widetilde{S}'_s(\xi)=c_2(s)\frac{\xi}{\abs{\xi}^2}\frac{1}{J_{s-1}^2(-i\abs{\xi})},
\end{equation}
	where $c_2(s)=2c_1(s)\sin(s\pi)/\pi$, and this also proves that $S_s(\xi)$ is radially monotone increasing in $\xi$. 
	
	Finally, from the properties of the Bessel function, we also know that $\widetilde{S}_s(\xi)$ is bounded, and we easily deduce that
	\[
	\lim_{\abs{\xi}\to+\infty}\frac{S_s(\xi)}{\abs{\xi}^{2s}}=C_2,
	\]
	where $C_2$ is a positive constant depending only on $s$. This proves~\eqref{asympt} and finishes the proof of Theorem \ref{thm_fourier}.
	\end{proof}
		
	We observe that if we take $a=0$ in~\eqref{anon0}, then $ c_1(1/2)=i $, and
		\begin{equation*}
			J_{-1/2}(-i\abs{\xi})=\frac{1+i}{\sqrt{\pi\abs{\xi}}}\cosh(\abs{\xi})
			\qquad 
			J_{1/2}(-i\abs{\xi})=\frac{1-i}{\sqrt{\pi\abs{\xi}}}\sinh(\abs{\xi}).
		\end{equation*}
		Therefore
		\[
		\widehat{\mathcal{L}_0u}(\xi)=\frac{e^{\lvert\xi\rvert}-e^{-\lvert\xi\rvert}}{e^{\lvert\xi\rvert}+e^{-\lvert\xi\rvert}}\abs\xi\widehat{u}(\xi),
		\]
		and we recover the special case~\eqref{a0}.
		
	To conclude this section, we deduce Corollary \ref{cor_equivalence} from Theorem \ref{thm_gammaconvergence}, providing an alternative form of the Dirichlet energy $\mathcal{E}_K$ associated to $\mathcal{L}_a$. 
		
	\begin{proof}[Proof of Corollary \ref{cor_equivalence}]
		Using the integration by parts formula and the fact
		that $v$ is a weak solution of \eqref{lvsis}, we have
		\begin{equation*}
		\mathcal{E}_K(v)=\frac12\int_{\R^n\times(0,1)}y^a\nabla v\cdot\nabla v\,dx\,dy=\frac12\int_{\R^n\times\{y=0\}}u\,\mathcal{L}_a(u)\,dx.
		\end{equation*}
		Applying Plancherel theorem and formula \eqref{anon0} for the Fourier transform of~$\mathcal{L}_a(u)$, we conclude that
		\begin{equation*}
		\mathcal{E}_K(v)=\frac1{2(2\pi)^n}\int_{\R^n}\widehat{\mathcal{L}_au}(\xi)\overline{\widehat{u}(\xi)}\,d\xi=\frac1{2(2\pi)^n}\int_{\R^n}S_s(\xi)\abs{\widehat{u}(\xi)}^2\,d\xi,
		\end{equation*}
	that concludes the proof of Corollary \ref{cor_equivalence}.	
\end{proof}

\section{$\Gamma$-convergence for $s\in[1/2,1)$}\label{sec_gamma_local}
This section is mainly devoted to the proof of Theorem \ref{thm_gammaconvergence} in the case $s\geq1/2$, that concerns the $\Gamma$-convergence of the functional $\mathcal{F}_\e$ defined in~\eqref{functional}. 

In the proof of the $\Gamma$-convergence result for $s\geq1/2$ we use the following Lemma~\ref{lemma_difference}
which establishes that the difference between the rescaled $H^s$ seminorm and the Dirichlet energy functional associated to the operator $\mathcal{L}_a$ is finite for every $u\in L^1(\R^n)$.
This result is valid for all~$s\in(0,1)$ and it will turn out to be useful
not only when~$s\in[1/2,1)$
to prove Theorem \ref{thm_gammaconvergence},
but 
also  when $s\in(0,1/2)$
to ensure that $\mathcal{F}$ is well-defined by \eqref{nonlocal_limit}.

\begin{lemma}\label{lemma_difference}
	For every $s\in(0,1)$ and $u\in L^1(\R^n)$, there exists a positive constant~$C$ depending only on $n$ and $s$ such that
	\[
	\int_{\R^n}\left(\widetilde{S}_s(\xi)-\mathcal{C}_s\right)\abs{\xi}^{2s}\abs{\widehat{u}(\xi)}^2\,d\xi\leq C\norm{u}^2_{L^1(\R^n)},
	\]
	where
	\begin{equation}\label{limit_Stilde}
	\mathcal{C}_s:=\lim_{\xi\to+\infty}\widetilde{S}_s(\xi)= 2^{1-2s}\frac{\Gamma(1-s)}{\Gamma(s)}.
	\end{equation}
\end{lemma}

\begin{proof}
	First, we observe that $\norm{\widehat{u}}_{L^\infty(\R^n)}\leq C \norm{u}_{L^1(\R^n)}$, for some positive constant $C$ depending only on $n$. Therefore, we have that
	\[
	\int_{\R^n}\left(\widetilde{S}_s(\xi)-\mathcal{C}_s\right)\abs{\xi}^{2s}\abs{\widehat{u}(\xi)}^2\,d\xi\leq C\norm{u}^2_{L^1(\R^n)}\int_{\R^n}\left(\widetilde{S}_s(\xi)-\mathcal{C}_s\right)\abs{\xi}^{2s}\,d\xi,	
	\]
	and we want to show that the integral in the right-hand side is finite.
	
	If $s=1/2$, the expression of $\widetilde{S}_{1/2}(\xi)$ is simpler, and one can directly check that $\mathcal{C}_{1/2}=1$ and
	\begin{equation*}
	\int_{\R^n}\left(\frac{e^{\lvert\xi\rvert}-e^{-\lvert\xi\rvert}}{e^{\lvert\xi\rvert}+e^{-\lvert\xi\rvert}}-1\right)\abs{\xi}\,d\xi= C\in(0,+\infty),
	\end{equation*}
	where the constant $C$ depends only on $n$.
	
For the general case of any $s\in(0,1)$, we want to show that there exists a positive constant~$C$ depending only on $n$ and $s$ such that
	\begin{equation}\label{finite:int}
	\int_{\R^n}\left(\widetilde{S}_s(\xi)-\mathcal{C}_s\right)\abs{\xi}^{2s}\,d\xi= C\in(0,+\infty).
	\end{equation}
	To this end, we can use polar coordinates and write the integral as
	\begin{equation}\label{INT:diff}
	\begin{split}
	&\int_{\R^n}\left(\widetilde{S}_s(\xi)-\mathcal{C}_s\right)\abs{\xi}^{2s}\,d\xi
	=
	\omega_{n-1}\int_{0}^{+\infty}\big(\widetilde{S}_s(r)-\widetilde{S}_s(+\infty)\big)r^{n-1+2s}\,dr
	\\
	&\hspace{1.3cm}\leq\omega_{n-1}\int_{0}^{+\infty}r^{n-1+2s}\,dr\int_{r}^{+\infty}\abs{\widetilde{S}'_s(t)}\,dt
	\\
	&\hspace{1.3cm}=\omega_{n-1}\int_{0}^{+\infty}\abs{\widetilde{S}'_s(t)}\,dt\int_{0}^{t}r^{n-1+2s}\,dr=\frac{\omega_{n-1}}{n+2s}\int_{0}^{+\infty}t^{n+2s}\abs{\widetilde{S}'_s(t)}\,dt.
	\end{split}
	\end{equation}
	Using~\eqref{der_Stilde} to compute $\widetilde{S}'_s(t)$, from \eqref{INT:diff} we deduce
	\begin{equation*}
	\int_{\R^n}\left(\widetilde{S}_s(\xi)-\mathcal{C}_s\right)\abs{\xi}^{2s}\,d\xi
	\leq 
	C\,c_2(s)\int_{0}^{+\infty}\frac{t^{n-1+2s}}{J^2_{s-1}(-it)}\,dt,
	\end{equation*}
	where $C$ is a positive constant depending only on $n$ and $s$. 
	Finally, the last integral is finite, since the integrand is bounded, and goes to zero at infinity faster than every power. This shows \eqref{finite:int} and concludes the proof of Lemma~\ref{lemma_difference}.
\end{proof}

Before proving Theorem \ref{thm_gammaconvergence} for $s\geq1/2$, we recall the setting in~\cite{SV} used by Savin and the second author to state their $\Gamma$-convergence result. Indeed, we prove Theorem \ref{thm_gammaconvergence} for $s\geq1/2$ by showing that the difference between the rescaled Dirichlet energies goes to zero at the limit, and then applying~\cite[Theorem~1.4]{SV}. 

We recall that the energy functional considered in~\cite{SV} is $\mathcal{I}_\e$ defined in~\eqref{SV_functional}. After a rescaling, we can assume that the double-well potential $V$ in~\eqref{SV_functional} satisfies~\eqref{double_well_potential},
and that the function space in~\cite{SV} is defined as $Y:=\{u\in L^\infty(\R^n): 0\leq u\leq1\}$. Following \cite{SV}, we say that $u_\e$ converges to $u$ in $Y$ if $u_\e\to u$ in $L^1_\textnormal{loc}(\R^n)$. 

Observe that our function space $X$ is contained in $Y$ and $X$ is equipped with the convergence in $L^1(\R^n)$. Thus, every time we consider a function $u$ in $X$ and a sequence $u_\e$ converging to $u$ in $X$, we are also in the setting considered in~\cite{SV}, and thus we are able to
exploit useful results from the existing literature. 

In respect to this matter,
we recall that in~\cite{SV} the functional~$\mathcal{I}_\e$
in~\eqref{SV_functional} is rescaled as
\begin{equation}\label{SV_rescaled_functional}
\mathcal{G}_\e(u,\Omega):=\begin{cases}
\begin{aligned}
&\e^{-2s}\mathcal{I}_\e(u,\Omega) \qquad &\text{if}\,\,s\in(0,1/2); \\
&\abs{\e\log\e}^{-1}\mathcal{I}_\e(u,\Omega) \qquad &\text{if}\,\,s=1/2; \\
&\e^{-1}\mathcal{I}_\e(u,\Omega) \qquad &\text{if}\,\,s\in(1/2,1).
\end{aligned}
\end{cases}
\end{equation}
Theorem 1.4 in~\cite{SV} establishes that $\mathcal{G}_\e$ converges in the $\Gamma$-sense to the classical perimeter if~$ s\in\left[1/2,1\right) $ and to the nonlocal area functional if~$ s\in\left(0,1/2 \right)$. More precisely, the $\Gamma$-limit functional in~\cite{SV} is  defined for $s\in(0,1/2)$ as 
\begin{equation*}
\mathcal{G}(u,\Omega):=\begin{cases}
\begin{aligned}
&\mathcal{K}(u,\Omega) \qquad &\text{if}\,\,u_{|\Omega}=\chi_E\,\,\text{for some set}\,\,E\subset\Omega; \\
&+\infty \qquad &\text{otherwise},
\end{aligned}
\end{cases}
\end{equation*} 
and for $s\in[1/2,1)$ as
\begin{equation}\label{SV_gamma_limit}
\mathcal{G}(u,\Omega):=\begin{cases}
\begin{aligned}
&c_*\textnormal{Per}(E,\Omega) \qquad &\text{if}\,\,u_{|\Omega}=\chi_E\,\,\text{for some set}\,\,E\subset\Omega; \\
&+\infty \qquad &\text{otherwise},
\end{aligned}
\end{cases}
\end{equation}
where $c_*$ is a constant depending only on $n$, $s$ and the double-well potential $V$ --- see~\cite{SV} for more details.

We are now able to prove Theorem~\ref{thm_gammaconvergence} for $s\geq1/2$.

\begin{proof}[Proof of Theorem~\ref{thm_gammaconvergence} for $\mathit{s\in[1/2,1)}$.] 
	First, considering the functional $\mathcal{F}_\e$ defined in~\eqref{functional}, we introduce the following notation for the $\e$-weights
\begin{equation*}
\lambda(\e):=
\begin{cases}
\begin{aligned}
&\abs{\log\e}^{-1} \qquad &\text{if}\,\,s=1/2; \\
&\e^{2s-1}\qquad &\text{if}\,\,s\in\left(1/2,1\right);
\end{aligned}
\end{cases}
\end{equation*}
and
\begin{equation*}
\kappa(\e):=
\begin{cases}
\begin{aligned}
&\abs{\e\log\e}^{-1} \qquad &\text{if}\,\,s=1/2; \\
&\e^{-1}\qquad &\text{if}\,\,s\in\left(1/2,1\right).
\end{aligned}
\end{cases}
\end{equation*}
Observe that the same $\e$-weights appear in the functional $\mathcal{G}_\e$ defined in~\eqref{SV_functional}, which is treated in~\cite{SV}. In this proof, we will exploit several times the fact that $\lambda(\e)\to0$ as $\e\to0^+$.

We recall that the square of the $H^{s}$-seminorm can be written as
\[
\seminorm{u}^2_{H^{s}(\R^n)}=
\iint_{\R^n\times\R^n}\frac{\abs{u(x)-u(y)}^2}{\abs{x-y}^{n+2s}}\,dx\,dy
=\frac{2C(n,s)^{-1}}{(2\pi)^n}\int_{\R^n}\abs{\xi}^{2s}\abs{\widehat{u}(\xi)}^2\,d\xi,
\]
where $C(n,s)$ is defined in~\eqref{C(n,s)}.

We consider $\mathcal{F}_\e(u_\e)$ and we use
the notation in~\eqref{S_s_tilde}. The limit at infinity of $ \widetilde{S}_s $ is denoted with $\mathcal{C}_s$ --- see~\eqref{limit_Stilde} and also Figure~\ref{figure_S_tilde} --- and it is finite and positive for every $s\in(0,1)$, then in particular in our case. 
We define
\begin{equation}\label{constant}
\overline{C}_s:=2^{n-1}\pi^n C(n,s)\,\mathcal{C}_s,
\end{equation} 
and we add and subtract $ \lambda(\e)\overline{C}_s\seminorm{u}^2_{H^s(\R^n)} $ to $\mathcal{F}_\e(u_\e)$. In this way we obtain
\begin{equation*}
\begin{split}
&\hspace{-0.7cm}\mathcal{F}_\e(u_\e)=\lambda(\e)\int_{\R^n}\abs{\xi}^{2s}\left(	\widetilde{S}_s(\xi)-\mathcal{C}_s	\right)\abs{\widehat{u}_\e(\xi)}^2\,d\xi
\\
&\hspace{1.5cm}+\lambda(\e)\overline{C}_s\iint_{\R^n\times\R^n}\frac{\abs{u_\e(x)-u_\e(y)}^2}{\abs{x-y}^{n+2s}}\,dx\,dy
+\kappa(\e)\int_{\R^n}W(u_\e)\,dx.
\end{split}
\end{equation*}
Using Lemma \ref{lemma_difference} and the fact that $\lambda(\e)\to0$ as $\e\to0^+$,  we deduce that for every $ u\in X$ and for every sequence $ (u_\e)_\e $ converging to $u$ in $L^1(\R^n)$, it holds that 
\[
\lim_{\e\to0^+}\lambda(\e)\int_{\R^n}\abs{\xi}^{2s}\left(	\widetilde{S}_s(\xi)-\mathcal{C}_s	\right)\abs{\widehat{u}_\e(\xi)}^2\,d\xi=0,
\]
Therefore, for every $u\in X$, if~$u_\e\to u$ in $L^1(\R^n)$, we have that
\begin{equation}\label{F_tilde}
\lim_{\e\to0^+}\left(\mathcal{F_\e}(u_\e)-\widetilde{\mathcal{F}}_\e(u_\e)\right)=0, 
\end{equation}
where
\begin{equation}\label{4.6bis}
\begin{split}
\widetilde{\mathcal{F}}_\e(w)&:=\lambda(\e)\,\mathcal{C}_s\int_{\R^n}\abs{\xi}^{2s}\abs{\widehat{w} (\xi)}^2\,d\xi
+\kappa(\e)\int_{\R^n}W(w)\,dx
\\
&=\lambda(\e)\,\overline{C}_{s}\iint_{\R^n\times\R^n}\frac{\abs{w(x)-w(y)}^2}{\abs{x-y}^{n+2s}}\,dx\,dy
+\kappa(\e)\int_{\R^n}W(w)\,dx,
\end{split}
\end{equation}
and $\overline{C}_{s}$ is defined in~\eqref{constant}.

Now, we use \eqref{F_tilde} and the $\Gamma$-convergence result in~\cite{SV} to deduce
the claims in~(i) and~(ii) of
Theorem~\ref{thm_gammaconvergence}. To this end,
we start from the liminf inequality in~(i).
 
For every function $u\in X$ we can choose a radius $R>0$ such that the ball $B_R\subset\R^n$ contains the support of $u$. Moreover, for any sequence $(u_\e)_\e$ that converges to $u$ in $L^1(\R^n)$, from Theorem 1.4 in~\cite{SV} we know that 
\begin{equation}\label{SV_liminf}
\liminf_{\e\to0^+}\mathcal{G}_\e(u_\e,B_R)\geq \mathcal{G}(u,B_R),
\end{equation}
where $\mathcal{G}(u,\Omega)$ and $\mathcal{G}_\e(u,\Omega)$ are defined respectively in~\eqref{SV_gamma_limit} and \eqref{SV_rescaled_functional}. In addition, by the definition of~$\widetilde{\mathcal{F}}_\e$
in~\eqref{4.6bis},
for every~$R>0$, we have that
\[
\iint_{\R^n\times\R^n}\frac{\abs{u_\e(x)-u_\e(y)}^2}{\abs{x-y}^{n+2s}}\,dx\,dy\geq \mathcal{K}(u,B_R),
\]
where $\mathcal{K}(u,B_R)$ appears in the definition of $\mathcal{G}_\e(u,B_R)$ given in~\eqref{SV_rescaled_functional}. In particular, it follows that
\begin{equation}\label{Gamma_89}
\widetilde{\mathcal{F}}_\e(u_\e)\geq\overline{C}_{s}\,\mathcal{G}_\e(u_\e,B_R),
\end{equation}
where $B_R\subset\R^n$ is the ball of radius $R$ containing the support of $u$.
We observe that both $\widetilde{\mathcal{F}}_\e$ and $\mathcal{G}_\e$ contain a double-well potential, and without loss of generality we can assume that
\begin{equation}\label{potentials}
W=\overline{C}_{s}V, 
\end{equation}
where $V$ is the potential function
in the definition of $\mathcal{G}_\e$ (recall~\eqref{SV_functional}
and~\eqref{SV_rescaled_functional}).

Then, using~\eqref{F_tilde}, \eqref{SV_liminf}, and \eqref{Gamma_89}, it follows that 
\begin{equation*}
\begin{split}
&\liminf_{\e\to0^+}\mathcal{F_\e}(u_\e)=\liminf_{\e\to0^+}\widetilde{\mathcal{F}}_\e(u_\e)
\\
&\hspace{3.5cm}
\geq\overline{C}_{s}\,\liminf_{\e\to0^+}\mathcal{G}_\e(u_\e,B_R)\geq \overline{C}_{s}\,\mathcal{G}(u,B_R)=\mathcal{F}(u),
\end{split}
\end{equation*}
which is the liminf inequality~\eqref{liminf} for a sequence $(u_\e)_\e$ converging to $u$ in $X$.
\medskip

Now, we prove the limsup inequality in claim~(ii)
of Theorem~\ref{thm_gammaconvergence}. For this,
we can assume that  \begin{equation}\label{FINITE}{\mbox{$ u=\chi_E $ for some set $ E\subset\R^n $, and $
\mathcal{F}(u)<+\infty$,}}\end{equation} otherwise the claim
in~(ii) is automatically satisfied.
 
In light of the definition of $X$ given in~\eqref{1.19bis},
since $u$ has compact support in $\R^n$, we can choose $R>2$ large enough such that 
\begin{equation}\label{SUPPO}
{\mbox{the support of $u$ is compactly contained in $B_{R/2}$.}}\end{equation} Moreover, from Theorem 1.4 in~\cite{SV} we know the existence of a sequence $u_\e$ that converges to $u$ in $B_{R}$ such that
\begin{equation}\label{SV_limsup}
\limsup_{\e\to0^+}\overline{C}_{s}\,\mathcal{G}_\e(u_\e,B_{R})\leq\overline{C}_{s}\,\mathcal{G}(u,B_{R})=\mathcal{F}(u),
\end{equation}
where the last equality follows from the definitions of $\mathcal{G}$ and $\mathcal{F}$, the fact that $u=\chi_E$, and that the support of $u$ is contained in $B_R$.

Besides, since $u_\e$ converges to $u$ in $L^1(B_{R})$, for every $k$ there exists $\e_k\in(0,1/k)$ such that
\begin{equation}\label{5_13}
\int_{B_{R}}\abs{u-u_{\e_k}}\,dx\leq\frac1k.
\end{equation}
In view of~\eqref{SV_limsup} we can also suppose that
\begin{equation}\label{5_56}
\overline{C}_{s}\,\mathcal{G}_{\e_k}(u_{\e_k},B_{R})\leq\mathcal{F}(u)+\frac1k.
\end{equation}
Now, for every $k\in\mathbb{N}\setminus\{0\}$, we define 
\begin{equation}\label{RHok}
\rho_k:=\frac1{k R^{n-1}}
\end{equation}
and
\begin{equation*}
u_k^*:=u_{\e_k}\psi_k,
\end{equation*}
where $\psi_k$ is a smooth function defined on $\R^n$ with values in $[0,1]$, such that 
\begin{equation}\label{psiK}
{\mbox{$\psi_k\equiv1$ in $B_{R-\rho_k}$,}}\qquad
{\mbox{$\psi_k\equiv0$ outside $B_{R}$,}}\qquad
{\mbox{and}}\qquad |\nabla\psi_k|\leq \frac{C}{\rho_k}.\end{equation}
Then, $u_k^*\in X $, and 
we claim that
\begin{equation}\label{8IAS}
{\mbox{$u_k^*$ converges to $u$ in $ L^1(\R^n) $}}.\end{equation}
Indeed, using \eqref{5_13} and that the support of $u$ is contained in $B_R$, we know that
\begin{equation*}
\begin{split}
\int_{\R^n}\abs{u_k^*-u}\,dx
&=
\int_{B_{R}-\rho_k}\abs{u_{\e_k}-u}\,dx+
\int_{B_{R}\setminus B_{R}-\rho_k}\abs{u_{\e_k}\psi_k-u}\,dx
\\
&\le\frac1k+
2\big|B_{R}\setminus B_{R-\rho_k}\big|
\\
&\le\frac1k+C\big( R^n-(R-\rho_k)^n\big),
\end{split}
\end{equation*}
for some~$C>0$ depending only on $n$.

This and~\eqref{RHok} yield that
\begin{equation*} \int_{\R^n}\abs{u_k^*-u}\,dx\le\frac1k+
CR^n \left(1- \left(1-\frac{\rho_k}{R}\right)^n\right)\le \frac1k+CR^{n-1}\rho_k=\frac{C}{k}.
\end{equation*}
{F}rom this, we plainly obtain~\eqref{8IAS}, as desired.

Now, we recall that $$ \limsup_{k\to+\infty}\mathcal{G}_{\e_k}(u_{\e_k},B_{R})<+\infty,$$ thanks to~\eqref{SV_limsup} and the assumption in~\eqref{FINITE}. We claim that
\begin{equation}\label{5_claim}
\limsup_{k\to\infty}\overline{C}_{s}\,\mathcal{G}_{\e_k}(u_{\e_k},B_{R})\geq\limsup_{k\to\infty}\widetilde{\mathcal{F}}_{\e_k}(u^*_k).
\end{equation}

To this end, recalling also \eqref{potentials}, we
observe that
\begin{equation}\label{difference}
\begin{split}
&\overline{C}_{s}\,\mathcal{G}_{\e_k}(u_{\e_k},B_{R})-\widetilde{\mathcal{F}}_{\e_k}(u^*_k)
\\
&\hspace{0.6cm}=I_k+I\!I_k+I\!I\!I_k+I\!V_k
+\kappa(\e_k)\int_{B_{R}\setminus B_{R-\rho_k}}\left(W(u_{\e_k})-W(u_{\e_k}\psi_k)\right)\,dx,
\end{split}
\end{equation}
where $I_k$, $I\!I_k$, $I\!I\!I_k$, and $I\!V_k$ are defined as
\begin{equation*}
\begin{split}
&I_k:=
2\overline{C}_{s}\lambda(\e_k)\iint_{B_{R-\rho_k}\times \left(B_{R}\setminus B_{R-\rho_k}\right)}\frac{\left(u_{\e_k}(x)-u_{\e_k}(y)\right)^2	}{\abs{x-y}^{n+2s}}
\\
&\hspace{7.4cm}-\frac{\left(u_{\e_k}(x)-\psi_k(y)u_{\e_k}(y)\right)^2}{\abs{x-y}^{n+2s}}\,dx\,dy;
\\
&I\!I_k:=2
\overline{C}_{s}\lambda(\e_k)\iint_{B_{R-\rho_k}\times\mathscr{C}B_{R}}\frac{u_{\e_k}(y)\left(u_{\e_k}(y)-2u_{\e_k}(x)	\right)	}{\abs{x-y}^{n+2s}}\,dx\,dy;
\\
&I\!I\!I_k:=2\overline{C}_{s}\lambda(\e_k)\iint_{(B_{R}\setminus B_{R-\rho_k})\times\mathscr{C}B_{R}}
\frac{\left(u_{\e_k}(x)-u_{\e_k}(y)\right)^2-u^2_{\e_k}(x)\psi^2_k(x)}{\abs{x-y}^{n+2s}}
\,dx\,dy;
\\
&I\!V_k:=\overline{C}_{s}\lambda(\e_k)\iint_{(B_{R}\setminus B_{R-\rho_k})^2}\frac{\left(u_{\e_k}(x)-u_{\e_k}(y)\right)^2}{\abs{x-y}^{n+2s}}
\\
&\hspace{6.4cm}
-\frac{\left(u_{\e_k}(x)\psi_k(x)-u_{\e_k}(y)\psi_k(y)\right)^2}{\abs{x-y}^{n+2s}}
\,dx\,dy.
\end{split}
\end{equation*}
First, we consider the difference of the potential energies in~\eqref{difference}. We claim that
\begin{equation}\label{potential_en}
\limsup_{k\to\infty}\,\kappa(\e_k)\int_{B_{R}\setminus B_{R-\rho_k}}
\left(W(u_{\e_k})-W(u_{\e_k}\psi_k)\right)\,dx\geq 0.
\end{equation}
To show this, first we recall that we are assuming that $\mathcal{F}(u)$ is finite, therefore $u=\chi_E$ for some set $E\subset\R^n$. We also remind that the recovery sequence $(u_\e)_\e$ is defined in~\cite{SV} as
\[
u_\e:=u_0\left(\frac{\textnormal{dist}(x)}{\e}\right),
\] 
where $u_0$ is the heteroclinic connecting the zeros of the potential $W$, i.e.\,0 and 1, and $\textnormal{dist}(x)$ is the signed distance of $x$ to $\partial E$, with the convention that $\textnormal{dist}(x)\geq0$ inside $E$ and $\textnormal{dist}(x)\leq0$ outside $E$
(see in particular~\cite[page~497]{SV}).

We remark that, in view of~\eqref{FINITE} and~\eqref{SUPPO}, we have that~$\chi_E=u=0$ outside~$B_{R/2}$,
hence~$E\subseteq B_{R/2}$.
In particular, if~$x$ lies outside~$B_{3R/4}$, we have that ${\textnormal{dist}(x)}\le -R/4$.
Hence, for $k$ big enough, we can assume that $u_{\e_k}$ is arbitrarily close to zero in $B_{R}\setminus B_{R-\rho_k}$. On the other hand, since $W$ is a double-well potential --- see \eqref{double_well_potential} --- it follows that $W'(t)\geq0$ for $t$ near zero. Therefore, since $u_{\e_k}\psi_k\leq u_{\e_k}$, for $k$ big enough we have that
\[
W(u_{\e_k})-W(u_{\e_k}\psi_k)\geq0\qquad\quad\text{in}\quad B_{R}\setminus B_{R-\rho_k},
\]
and this shows \eqref{potential_en}.

Considering now the integral in~$I_k$ in~\eqref{difference}, we observe that 
\begin{equation*}
\begin{split}
&\hspace{-1cm}\left(u_{\e_k}(x)-u_{\e_k}(y)\right)^2-\left(u_{\e_k}(x)-\psi_k(y)u_{\e_k}(y)\right)^2
\\
&\hspace{1.5cm}=\left(1-\psi^2_k(y)\right)u_{\e_k}^2(y)-2u_{\e_k}(x)u_{\e_k}(y)\left(1-\psi_k(y)\right)
\\
&\hspace{1.5cm}=u_{\e_k}(y)(1-\psi_k(y))\big((1+\psi_k(y))u_{\e_k}(y)-2u_{\e_k}(x)\big)
\end{split}
\end{equation*}
Since in this case
we are integrating $x$ over $B_{R-\rho_k}$, we have that $\psi_k(x)=1$. Hence,
we can write that
\begin{equation*}
\begin{split}
&\hspace{-1cm}(1-\psi_k(y))\big((1+\psi_k(y))u_{\e_k}(y)-2u_{\e_k}(x)\big)
\\
&=
(\psi_k(x)-\psi_k(y))\big((2-\psi_k(x)+\psi_k(y))u_{\e_k}(y)-2u_{\e_k}(x)\big)
\\
&=(\psi_k(x)-\psi_k(y))\big(
2(u_{\e_k}(y)-u_{\e_k}(x))
-(\psi_k(x)-\psi_k(y))u_{\e_k}(y)\big)
\\
&=2(\psi_k(x)-\psi_k(y))(u_{\e_k}(y)-u_{\e_k}(x))
-(\psi_k(x)-\psi_k(y))^2u_{\e_k}(y) .
\end{split}
\end{equation*}
Consequently,
using also that $u_{\e_k}$ is uniformly bounded, we see that
\begin{equation}\label{5_67}
\begin{split}
I_k\leq C\lambda(\e_k)\iint_{B_{R-\rho_k}\times \left(B_{R}\setminus B_{R-\rho_k}\right)}\frac{\left(\psi_k(x)-\psi_k(y)	\right)\left(u_{\e_k}(x)-u_{\e_k}(y)	\right)}{\abs{x-y}^{n+2s}}\,dx\,dy
\\
+C\lambda(\e_k)\iint_{B_{R-\rho_k}\times \left(B_{R}\setminus B_{R-\rho_k}\right)}\frac{\left(\psi_k(x)-\psi_k(y)\right)^2}{\abs{x-y}^{n+2s}}\,dx\,dy.
\end{split}
\end{equation}
To estimate the second integral in~\eqref{5_67}, we use~\eqref{psiK}
to deduce that $$\left(\psi_k(x)-\psi_k(y)\right)^2\leq \frac{C}{\rho_k^2}\abs{x-y}^2,$$ and we obtain that
\begin{equation}\label{LAs-2}
\begin{split}
&\iint_{B_{R-\rho_k}\times \left(B_{R}\setminus B_{R-\rho_k}\right)}\frac{\left(\psi_k(x)-\psi_k(y)\right)^2}{\abs{x-y}^{n+2s}}\,dx\,dy
\\
&\hspace{2cm}\leq
\frac{C}{\rho_k^2}\iint_{B_{R-\rho_k}\times \left(B_{R}\setminus B_{R-\rho_k}\right)}\frac{1}{\abs{x-y}^{n+2s-2}}\,dx\,dy=:\mu_k\in(0,+\infty)
.
\end{split}
\end{equation}
Observe that $\mu_k$ is finite since $s\in[1/2,1)$ and $\abs{x-y}^{-n-2s+2}$ is integrable.
Accordingly, we can choose $\e_k$ so small that $\lambda(\e_k)\le (k\mu_k)^{-1}$ and we conclude that
\begin{equation}\label{LAs-3}
\begin{split}
&\hspace{-1.5cm}\lim_{k\to+\infty}
C\lambda(\e_k)\iint_{B_{R-\rho_k}\times \left(B_{R}\setminus B_{R-\rho_k}\right)}\frac{\left(\psi_k(x)-\psi_k(y)\right)^2}{\abs{x-y}^{n+2s}}\,dx\,dy
\\ 
&\hspace{3.5cm}\le
\lim_{k\to+\infty}C\lambda(\e_k)\mu_k\le\lim_{k\to+\infty}
\frac{C}{k}
=0.\end{split}
\end{equation}
This controls the second integral
in~\eqref{5_67}.
Instead,
for the first integral in~\eqref{5_67}, we can use the Cauchy-Schwarz inequality and \eqref{SV_limsup}, to write that
\begin{equation*}
\begin{split}
&\lambda(\e_k)\iint_{B_{R-\rho_k}\times \left(B_{R}\setminus B_{R-\rho_k}\right)}\frac{\left(\psi_k(x)-\psi_k(y)	\right)\left(u_{\e_k}(x)-u_{\e_k}(y)	\right)}{\abs{x-y}^{n+2s}}\,dx\,dy
\\ 
&\hspace{0.5cm}\leq
\left(\lambda(\e_k)\iint_{B_{R-\rho_k}\times \left(B_{R}\setminus B_{R-\rho_k}\right)}\frac{\left(\psi_k(x)-\psi_k(y)	\right)^2}{\abs{x-y}^{n+2s}}\,dx\,dy\right)^\frac12\,
\\
&\hspace{2cm}\times
\left(\lambda(\e_k)\iint_{B_{R-\rho_k}\times \left(B_{R}\setminus B_{R-\rho_k}\right)}\frac{\left(u_{\e_k}(x)-u_{\e_k}(y)	\right)^2}{\abs{x-y}^{n+2s}}\,dx\,dy\right)^\frac12\\
&\hspace{0.5cm}\leq
\left(\lambda(\e_k)\iint_{B_{R-\rho_k}\times \left(B_{R}\setminus B_{R-\rho_k}\right)}\frac{\left(\psi_k(x)-\psi_k(y)	\right)^2}{\abs{x-y}^{n+2s}}\,dx\,dy\right)^\frac12
\left(\mathcal{F}(u)+\frac1k\right)^\frac12.
\end{split}
\end{equation*}
Hence, in view of~\eqref{FINITE} and~\eqref{LAs-2}-\eqref{LAs-3}, we write that
\begin{equation*}
\lim_{k\to+\infty}\lambda(\e_k)\iint_{B_{R-\rho_k}\times \left(B_{R}\setminus B_{R-\rho_k}\right)}\frac{\left(\psi_k(x)-\psi_k(y)	\right)\left(u_{\e_k}(x)-u_{\e_k}(y)	\right)}{\abs{x-y}^{n+2s}}\,dx\,dy
=0.
\end{equation*}
{F}rom this and~\eqref{LAs-3}, we conclude that
\begin{equation}\label{I_k_limit}
\lim_{k\to+\infty} I_k=0.\end{equation}
Considering now the integral in $I\!I_k$,
we exploit that $u_{\e_k}$ is uniformly bounded and that
\begin{equation*}
\iint_{B_{R-\rho_k}\times\mathscr{C}B_{R}}\frac{1}{\abs{x-y}^{n+2s}}\,dx\,dy
=:\widetilde\mu_k\in(0,+\infty).
\end{equation*}
In this way, we conclude that
$$ I\!I_k\leq C\,\widetilde\mu_k \lambda(\e_k).$$
Consequently, choosing $\e_k$ so small that $\lambda(\e_k)\leq (k \widetilde{\mu}_k)^{-1}$, we conclude that
\begin{equation}\label{II_k_limit}
\lim_{k\to+\infty}I\!I_k\leq\lim_{k\to+\infty}\frac{C}{k}=0.
\end{equation}

Now, we consider the integral in $I\!I\!I_k$ and we claim that
\begin{equation}\label{III_k}
\limsup_{k\to\infty}I\!I\!I_k\geq0.
\end{equation}
To this end, it is sufficient to show that
\begin{equation}\label{III_k_1}
\lim_{k\to+\infty}\lambda(\e_k)\iint_{(B_{R}\setminus B_{R-\rho_k})\times\mathscr{C}B_{R}}
\frac{u^2_{\e_k}(x)\psi^2_k(x)}{\abs{x-y}^{n+2s}}
\,dx\,dy=0,
\end{equation}
since the other part of the integral in $I\!I\!I_k$ is positive.
 
We use that $u_{\e_k}$ is uniformly bounded and that $\psi_k(y)=0$ since we are integrating $y$ over $\mathscr{C}B_{R}$, to write
that
\begin{equation}\label{III_k_2}
\begin{split}
&\hspace{-0.5cm}\lambda(\e_k)\iint_{(B_{R}\setminus B_{R-\rho_k})\times\mathscr{C}B_{R}}
\frac{u^2_{\e_k}(x)\psi^2_k(x)}{\abs{x-y}^{n+2s}}
\,dx\,dy
\\
&\hspace{0.5cm}\leq C\lambda(\e_k)\iint_{(B_{R}\setminus B_{R-\rho_k})\times(B_{R+1}\setminus B_{R})}
\frac{\left(\psi_k(x)-\psi_k(y)\right)^2}{\abs{x-y}^{n+2s}}
\,dx\,dy
\\
&\hspace{2cm}+C\lambda(\e_k)\iint_{(B_{R}\setminus B_{R-\rho_k})\times\mathscr{C}B_{R+1}}
\frac{1}{\abs{x-y}^{n+2s}}
\,dx\,dy.
\end{split}
\end{equation}
To control the first integral in the right-hand side of \eqref{III_k_2},
we use that $(\psi_k(x)-\psi_k(y))^2\leq C\abs{x-y}^2/\rho_k^2$, obtaining
\begin{equation*}
\begin{split}
&\iint_{(B_{R}\setminus B_{R-\rho_k})\times(B_{R+1}\setminus B_{R})}\frac{\left(\psi_k(x)-\psi_k(y)\right)^2}{\abs{x-y}^{n+2s}}\,dx\,dy
\\
&\hspace{2cm}\leq
\frac{C}{\rho_k^2}\iint_{(B_{R}\setminus B_{R-\rho_k})\times(B_{R+1}\setminus B_{R})}\frac{1}{\abs{x-y}^{n+2s-2}}\,dx\,dy=:\nu_k\in(0,+\infty)
.
\end{split}
\end{equation*}
Therefore, we can choose $\e_k$ so small that $\lambda(\e_k)\leq (k\nu_k)^{-1}$ and we deduce that
\begin{equation*}
\begin{split}
&\hspace{-1.5cm}\lim_{k\to+\infty}
C\lambda(\e_k)\iint_{(B_{R}\setminus B_{R-\rho_k})\times(B_{R+1}\setminus B_{R})}\frac{\left(\psi_k(x)-\psi_k(y)\right)^2}{\abs{x-y}^{n+2s}}\,dx\,dy
\\ 
&\hspace{3.5cm}\le
\lim_{k\to+\infty}C\nu_k\lambda(\e_k)\le\lim_{k\to+\infty}
\frac{C}{k}
=0.\end{split}
\end{equation*}
Concerning the last integral in~\eqref{III_k_2}, we integrate first $y$ over $\mathscr{C}B_{R+1}$, and then $x$ over $B_{R}\setminus B_{R-\rho_k}$, to obtain
\begin{equation*}
\begin{split}
&\hspace{-1cm}\lim_{k\to+\infty}C\lambda(\e_k)\iint_{(B_{R}\setminus B_{R-\rho_k})\times\mathscr{C}B_{R+1}}
\frac{1}{\abs{x-y}^{n+2s}}
\,dx\,dy
\\
&\hspace{1cm}\leq\lim_{k\to+\infty}C\lambda(\e_k)\abs{B_{R}\setminus B_{R-\rho_k}}
\leq\lim_{k\to+\infty}\frac{C}{k}\lambda(\e_k)=0.
\end{split}
\end{equation*}
This shows the validity of \eqref{III_k_1}, and concludes the proof of \eqref{III_k} about $I\!I\!I_k$.

Now, we consider the integral in $I\!V_k$. Using the expression
\[
u_{\e_k}(x)\psi_k(x)-u_{\e_k}(y)\psi_k(y)=u_{\e_k}(x)\left(\psi_k(x)-\psi_k(y)\right)+\psi_k(y)\left(u_{\e_k}(x)-u_{\e_k}(y)\right),
\]
we write $I\!V_k$ as
\begin{equation}\label{IV_k_1}
\begin{split}
&I\!V_k=\lambda(\e_k)\iint_{(B_{R}\setminus B_{R-\rho_k})^2}\frac{\left(1-\psi^2_k(y)\right)\left(u_{\e_k}(x)-u_{\e_k}(y)\right)^2}{\abs{x-y}^{n+2s}}\,dx\,dy
\\
&-\lambda(\e_k)\iint_{(B_{R}\setminus B_{R-\rho_k})^2}\frac{u^2_{\e_k}(x)\left(\psi_k(x)-\psi_k(y)	\right)^2}{\abs{x-y}^{n+2s}}\,dx\,dy
\\
&-\lambda(\e_k)\iint_{(B_{R}\setminus B_{R-\rho_k})^2}\frac{2u_{\e_k}(x)\psi_k(y)\left(u_{\e_k}(x)-u_{\e_k}(y)\right)\left(\psi_k(x)-\psi_k(y)\right)}{\abs{x-y}^{n+2s}}\,dx\,dy
\end{split}
\end{equation}
Since $\abs{\psi_k(y)}\leq1$, the first integral in the right-hand side of~\eqref{IV_k_1} is nonnegative. To control the second term, we use that $(\psi_k(x)-\psi_k(y))^2\leq C\abs{x-y}^2/\rho_k^2$, and write
\begin{equation*}
\begin{split}
&\iint_{(B_{R}\setminus B_{R-\rho_k})^2}\frac{\left(\psi_k(x)-\psi_k(y)	\right)^2}{\abs{x-y}^{n+2s}}\,dx\,dy
\\
&\hspace{2cm}
\leq C\rho_k^{-2}\iint_{(B_{R}\setminus B_{R-\rho_k})^2}\frac{1}{\abs{x-y}^{n+2s-2}}\,dx\,dy
=:\widetilde{\nu}_k\in(0,+\infty).
\end{split}
\end{equation*}
Hence, choosing $\e_k$ so small that $\lambda(\e_k)\leq (k\widetilde{\nu}_k)^{-1}$, it follows that
\begin{equation}\label{IV_k_2}
\lim_{k\to+\infty}\lambda(\e_k)\iint_{(B_{R}\setminus B_{R-\rho_k})^2}\frac{\left(\psi_k(x)-\psi_k(y)	\right)^2}{\abs{x-y}^{n+2s}}\,dx\,dy
\leq\lim_{k\to+\infty}\frac{C}{k}
=0.
\end{equation}
In the last integral in~\eqref{IV_k_1}, we exploit that $u_{\e_k}(x)\psi_k(y)$ is uniformly bounded and we use the Cauchy-Schwarz inequality to write
\begin{equation*}
\begin{split}
&\hspace{-1cm}\lambda(\e_k)\iint_{(B_{R}\setminus B_{R-\rho_k})^2}\frac{\left(u_{\e_k}(x)-u_{\e_k}(y)\right)\left(\psi_k(x)-\psi_k(y)\right)}{\abs{x-y}^{n+2s}}\,dx\,dy
\\
&\leq\left(\lambda(\e_k)\iint_{(B_{R}\setminus B_{R-\rho_k})^2}\frac{\left(u_{\e_k}(x)-u_{\e_k}(y)\right)^2}{\abs{x-y}^{n+2s}}\,dx\,dy\right)^\frac12
\\
&\hspace{2cm}\times\left(\lambda(\e_k)\iint_{(B_{R}\setminus B_{R-\rho_k})^2}\frac{\left(\psi_k(x)-\psi_k(y)\right)^2}{\abs{x-y}^{n+2s}}\,dx\,dy\right)^\frac12
\\
&\leq\left(\lambda(\e_k)\iint_{(B_{R}\setminus B_{R-\rho_k})^2}\frac{\left(\psi_k(x)-\psi_k(y)\right)^2}{\abs{x-y}^{n+2s}}\,dx\,dy\right)^\frac12\left(\mathcal{F}(u)+\frac1k\right)^\frac12.
\end{split}
\end{equation*}
Recalling~\eqref{FINITE} and~\eqref{IV_k_2}, we thereby see that
\begin{equation}\label{IV_k_3}
\lim_{k\to+\infty} \lambda(\e_k)\iint_{(B_{R}\setminus B_{R-\rho_k})^2}\frac{\left(u_{\e_k}(x)-u_{\e_k}(y)\right)\left(\psi_k(x)-\psi_k(y)\right)}{\abs{x-y}^{n+2s}}\,dx\,dy
=0.
\end{equation}
Now, putting together \eqref{IV_k_2}, \eqref{IV_k_3}, and the fact that the first integral in the right hand side of \eqref{IV_k_1} is positive, we deduce that
\begin{equation}\label{IV_k}
\limsup_{k\to\infty}I\!V_k\geq0.
\end{equation}

Finally, from~\eqref{I_k_limit}, \eqref{II_k_limit}, \eqref{III_k} and \eqref{IV_k}
we deduce the desired claim in~\eqref{5_claim}.
Now, in light of~\eqref{F_tilde}, \eqref{SV_limsup}, \eqref{5_56}, and \eqref{5_claim},
we have that
\[
\limsup_{k\to\infty}\mathcal{F}_k(u^*_k)=\limsup_{k\to\infty}\widetilde{\mathcal{F}}_{\e_k}(u^*_k)\leq\limsup_{k\to\infty}\overline{C}_{s}\,\mathcal{G}_{\e_k}(u_{\e_k},B_{R})\leq\mathcal{F}(u),
\]
that is the claim in~(ii) of Theorem \ref{thm_gammaconvergence}.

This completes the proof of Theorem \ref{thm_gammaconvergence} for $s\in[1/2,1)$.
For completeness, we observe that the constant $c_\#$ appearing in the $\Gamma$-limit \eqref{local_limit} can be written as 
\[c_\#=\overline{C}_s\,c_*, \]
where $\overline{C}_s$ is defined in \eqref{constant} and $c_*$ is the constant appearing in \eqref{SV_gamma_limit}, which in turn
is related to the $\Gamma$-limit functional in \cite{SV} for $s\geq1/2$.
\end{proof}

\section{$\Gamma$-convergence for $ s\in(0,1/2) $}\label{sec_nonlocal}
This section is focused on the $\Gamma$-convergence for the case $s\in(0,1/2)$. First, we prove Theorem \ref{thm_gammaconvergence} in this case, and then we prove Proposition~\ref{prop_asympt}.

\begin{proof}[Proof of Theorem~\ref{thm_gammaconvergence} for $\mathit{s\in(0,1/2)}$.]
	We consider any $u\in X$ and we start by proving the claim in (i), which is the liminf inequality for every sequence $u_\e$ converging to~$u$ in~$X$. 
	Let $u_\e$ be a sequence of functions in $X$ that converges to $u$ in $L^1(\R^n)$. If
	\[
	\liminf_{\e\to0^+}\mathcal{F}_\e(u)=+\infty,
	\]
	then \eqref{liminf} is obvious. Hence, we assume that
	\[
	\liminf_{\e\to0^+}\mathcal{F}_\e(u)=l<\infty.
	\]
	We take $(u_{\e_k})_k$ as a subsequence of $(u_{\e})$ that attains the limit $l$, and $(u_{\e_{k_j}})_j$ as a
	subsequence that converges to $u$ almost everywhere. Then,
	\[
	l=\liminf_{k\to\infty}\mathcal{F}_{\e_k}(u_{\e_k})=\lim_{j\to\infty}\mathcal{F}_{\e_{k_j}}(u_{\e_{k_j}})\geq\lim_{j\to\infty}\frac{1}{\e_{k_j}^{2s}}\int_{\R^n}W(u_{\e_{k_j}})\,dx.
	\]
	Therefore,
	\[
	\int_{\R^n} W(u)\,dx=\lim_{j\to+\infty}\int_{\R^n} W(u_{\e_{k_j}})\,dx=0,
	\]
	and $ u(x)\in\{0;1\} $ almost everywhere. Thus, we deduce that $ u=\chi_E $ for some set $ E\subset\R $. Using Fatou's lemma and the definition of $\mathcal{F}(u)$ in~\eqref{nonlocal_limit}, we can conclude that
	\[
	\liminf_{\e\to0^+}\mathcal{F}_\e(u_\e)\geq\liminf_{\e\to0^+}\int_{\R^n}S_s(\xi)\abs{\widehat{u}_\e(\xi)}^2\,d\xi\geq\int_{\R^n}S_s(\xi)\abs{\widehat{u}(\xi)}^2\,d\xi=\mathcal{F}(u).
	\]
	This completes the proof of the claim in~(i).
	
	Now, we prove the claim in (ii). We assume that~$ u=\chi_E $ for some set $ E\subset\R^n $ --- otherwise \eqref{limsup} is obvious --- and we define the constant sequence~$u_\e:=u$. 
	
	Since $\mathcal{F}_\e(u)$ is defined for $u=\chi_E$ as
	\[
	\mathcal{F}_\e(u)=\mathcal{F}(u)=\int_{\R^n} S_s(\xi)\abs{\widehat{u}(\xi)}^2\,d\xi,
	\]
	then we trivially have \eqref{limsup} for the constant sequence $(u_\e)_\e$.
\end{proof}

Now, we prove Proposition \ref{prop_asympt}. This result gives important information about the limit functional $ \mathcal{F} $ defined in~\eqref{nonlocal_limit} for $s\in(0,1/2)$ in the case $n=1$, showing that it interpolates the classical and the nonlocal perimeter.

\begin{proof}[Proof of Proposition~\ref{prop_asympt}]
We recall that the function $ \mathcal{T}_s(r):[0,+\infty)\longrightarrow[0,+\infty) $ is defined as
\begin{equation}\label{T_s}
\mathcal{T}_s(r):=\mathcal{F}(\chi_{I_r})=\int_\R S_s(\xi)\abs{\widehat{\chi_{I_r}}(\xi)}^2\,d\xi,
\end{equation}	
and the squared modulus of the Fourier transform of~$\chi_{I_r}$~is
\[
\abs{\widehat{\chi_{I_r}}(\xi)}^2=\frac{4\sin^2(r\xi)}{\xi^2}.
\]
This last computation is done in detail in Lemma~\ref{lemma_onebump} in the appendix. Since the squared modulus of $\widehat{\chi_{I_r}}$ depends only on the length of the interval, then $ \mathcal{F}(\chi_{I_r}) $ only depends on $r$ and $\mathcal{T}_s$ is a well-defined function of $r\in[0,+\infty)$. 

Plugging the expression of $\abs{\widehat{\chi_{I_r}}(\xi)}^2$ in~\eqref{T_s}, we have
\begin{equation}\label{T_s_1}
\mathcal{T}_s(r)=4\int_\R \widetilde{S}_s(\xi) \frac{\sin^2(r\xi)}{\abs{\xi}^{2-2s}}\,d\xi,
\end{equation}
where $\widetilde{S}_s(\xi)$ is defined in~\eqref{S_s_tilde}.

We want to show the asymptotic behavior of $\mathcal{T}_s$ at zero,
as stated in~\eqref{asynt_0}. To this end, we change variable $ r\xi=\eta $ in~\eqref{T_s_1} and we get the following expression for~$\mathcal{T}_s(r)$
\begin{equation}\label{T_s_change}
\mathcal{T}_s(r) = 4r^{1-2s}\int_\R \widetilde{S}_s\left(\frac{\eta}{r}\right)\frac{\sin^2(\eta)}{\abs{\eta}^{2-2s}}\,d\eta.
\end{equation}

{F}rom the dominated convergence theorem and the fact that $\sin^2(\eta)/\abs{\eta}^{2-2s}$ is integrable in $ \R $ when $ s\in(0,1/2) $, we deduce that
\[
\lim_{r\to0}\int_\R \widetilde{S}_s\left(\frac{\eta}{r}\right)\frac{\sin^2(\eta)}{\abs{\eta}^{2-2s}}\,d\eta=C_1,
\]
where $C_1$ is a positive constant depending only on $s$. Thus, from this bound and~\eqref{T_s_change} we obtain~\eqref{asynt_0}, as desired.

	Now, we want to prove~\eqref{asynt_infty}, which describes the asymptotic behavior of $\mathcal{T}_s$ at infinity. We use the expression
	in~\eqref{T_s_1} for $\mathcal{T}_s(r)$ and Lemma~\ref{limit_sin} to write
	\begin{equation}\label{432}
	\lim_{r\to\infty}\mathcal{T}_s(r)=\lim_{r\to\infty}4\int_\R\widetilde{S}_s(\xi) \frac{\sin^2(r\xi)}{\abs{\xi}^{2-2s}}\,d\xi=2\int_{\R}\frac{\widetilde{S}_s(\xi)}{\abs{\xi}^{2-2s}}\,d\xi.
	\end{equation}
	The function in the last integral is controlled by a constant near the origin --- see~\eqref{S_s_tilde} and \eqref{propJ} --- and by $C/\abs{\xi}^{2-2s}$ far from the origin, which is an integrable function at infinity, since $s\in(0,1/2)$. 
	
	Therefore, the last integral in~\eqref{432} is finite and this proves
	the desired claim in~\eqref{asynt_infty}. The proof of Proposition~\ref{prop_asympt} is thereby complete.
	\end{proof}

\begin{appendices}
	\section{Appendix}
	For the sake of completeness,
	we collect here two simple technical lemmata. Let us start with a very standard computation, that is the Fourier transform of the characteristic function of one interval.
	\begin{lemma}\label{lemma_onebump}
		Let $ u(x):\R\to[0,1] $ be defined as $ u(x)=\chi_I $, where $I$ is a finite interval of $\R$, i.e.\,$ I=(a_1,a_2)\subset\R $. Then, 
		\[
		\abs{\widehat{u}(\xi)}^2=4\frac{\sin^2\left(r\xi\right)}{\xi^2},
		\]
		where $ r=\frac{a_2-a_1}{2} $ is the width of the intervals $ I $.
	\end{lemma}

	\begin{proof}
	First, we compute the Fourier transform of the function $ u $.	
		\begin{equation*}
		\begin{split}
		&\widehat{u}(\xi)=\int_{a_1}^{a_2}e^{-ix\xi}\,dx=\frac{i}{\xi}\left(e^{-ia_2\xi}-e^{-ia_1\xi}\right) 
		\\
		&\hspace{3cm}= \frac{1}{\xi}\left\{\sin(a_2\xi)-\sin(a_1\xi)+i\left(\cos(a_2\xi)-\cos(a_1\xi)\right)	\right\}.
		\end{split}
		\end{equation*}
		Then, we compute its square modulus.
		\begin{equation*}
		\begin{split}
		&\abs{\widehat{u}(\xi)}^2=\frac{1}{\xi^2}\left\{2-2\left(\sin(a_2\xi)\sin(a_1\xi)+\cos(a_2\xi)\cos(a_1\xi)\right)\right\}	
		\\
		&\hspace{3cm}=\frac{1}{\xi^2}\left\{2-2\cos\left((a_2-a_1)\xi\right)\right\}=\frac{4}{\xi^2}\sin^2\left(\frac{a_2-a_1}{2}\xi\right),
		\end{split}
		\end{equation*}
		and this concludes the proof of Lemma \ref{lemma_onebump}.
	\end{proof}
	
	We prove now a convergence result that we use in Section~\ref{sec_nonlocal}.
	\begin{lemma}\label{limit_sin}
		If $ f\in L^1(\R) $, then
		\[
		\lim_{\omega\to+\infty}\int_\R f(\eta)\sin^2(\omega\eta)\,d\eta=\frac12\int_\R f(\eta)\,d\eta.
		\]
	\end{lemma}
	\begin{proof}
		Let us assume first that $f\in C^1_c(\R)$. We start from the identity
		\[
		\int_{\R}f(\eta)\,d\eta=\int_{\R}f(\eta)\sin^2\left(\omega\eta\right)\,d\eta+\int_{\R}f(\eta)\cos^2\left(\omega\eta\right)\,d\eta,
		\]
		and we want to show that
		\begin{equation}\label{79}
		\lim_{\omega\to+\infty}\int_{\R}f(\eta)\cos^2\left(\omega\eta\right)\,d\eta=\lim_{\omega\to+\infty}\int_{\R}f(\eta)\sin^2\left(\omega\eta\right)\,d\eta.
		\end{equation}
		We remark indeed that the claim in Lemma \ref{limit_sin} follows once we
		establish~\eqref{79}. In order to prove~\eqref{79}, we change variables $ \omega\eta=\omega\theta-\pi/2 $ and we obtain
		\begin{equation}\label{489}
			\begin{split}
			&\int_{\R}f(\eta)\cos^2\left(\omega\eta\right)\,d\eta=
			\\
			&\hspace{1.5cm}\int_{\R}f(\theta)\sin^2\left(\omega\theta\right)\,d\theta+\int_{\R}\left\{f\left(\theta-\frac{\pi}{2\omega}\right)-f(\theta)\right\}\sin^2\left(\omega\theta\right)\,d\theta.
			\end{split}
		\end{equation}
		
		Taking the limits as $ \omega\to+\infty $ in~\eqref{489}, the last term goes to zero thanks to the Vitali convergence theorem and we obtain~\eqref{79} if $f\in C^1_c(\R)$. In general, when $ f\in L^1(\R) $, the result follows from the density of $ C^1_c(\R) $ in $L^1(\R)$.
	\end{proof}

	\end{appendices}

\vspace{1mm}
\section*{Acknowledgment}
The authors would like to thank Xavier Cabré and Matteo Cozzi for very interesting and useful discussions
on the topic of this paper.

\vfill	
\end{document}